\documentclass[nohyperref]{article}

\usepackage[utf8]{inputenc}

\usepackage{csquotes}

\usepackage[english]{babel}
\usepackage{url}
\usepackage{mathtools}
\usepackage{enumitem}
\usepackage{dsfont}
\usepackage{bbm}

\usepackage{stmaryrd}
\usepackage[utf8]{inputenc}
\usepackage[T1]{fontenc}
\usepackage{placeins}
\usepackage{amsmath, amsfonts, amssymb}
\usepackage{graphicx}
\usepackage[nottoc]{tocbibind}

\usepackage{amsthm}
\newcommand\tp{\Tilde{P}}
\newcommand\JP{P^{\alpha,\beta}}

\newcommand\tmu{\Tilde{\mu}}

\theoremstyle{plain}
\newtheorem{theorem}{Theorem}[section]

\newtheorem{lemma}[theorem]{Lemma}

\theoremstyle{definition}
\newtheorem{definition}[theorem]{Definition}
\newtheorem*{assumption}{Assumption}
\theoremstyle{remark}
\newtheorem{remark}[theorem]{Remark}
\newtheorem{problem}[theorem]{Problem}
\newenvironment{sketch}{%
  \proof}{\endproof}
\usepackage[utf8]{inputenc} % allow utf-8 input
\usepackage[T1]{fontenc}    % use 8-bit T1 fonts
\usepackage{url}            % simple URL typesetting
\usepackage{booktabs}       % professional-quality tables
\usepackage{amsfonts}       % blackboard math symbols
\usepackage{nicefrac}       % compact symbols for 1/2, etc.
\usepackage{microtype}      % microtypography

\usepackage{appendix}

\usepackage{empheq}
\usepackage{xcolor, color, colortbl}
\usepackage{framed}
\usepackage{pifont}
\usepackage{enumitem}
\usepackage{graphicx} % more modern
\usepackage{nicefrac}
\usepackage{booktabs}
\usepackage{multirow}
\usepackage{rotating}
\usepackage{nccmath}

\usepackage[tikz]{bclogo}
\usepackage{wasysym}

\definecolor{mydarkblue}{rgb}{0,0.08,0.45}
\usepackage[
    colorlinks=true,
    linkcolor=mydarkblue,
    citecolor=mydarkblue,
    filecolor=mydarkblue,
    urlcolor=mydarkblue,
    pdfview=FitH]{hyperref}
\usepackage{url}
\usepackage{relsize}

\def\xx{{\boldsymbol x}}
\def\yy{{\boldsymbol y}}

\def\XX{{\boldsymbol X}}

\def\yy{{\boldsymbol y}}

\def\UU{{\boldsymbol U}}
\def\HH{{\boldsymbol H}}

\def\dif{\mathop{}\!\mathrm{d}}

\def\RR{{\mathbb R}}
\def\EE{{\mathbb{E}}}

\DeclareMathOperator{\tr}{tr}
\def\defas{\stackrel{\text{def}}{=}}

\DeclareMathOperator*{\Span}{\mathbf{span}}

\newcommand*\mybluebox[1]{\colorbox{myblue}{\hspace{1em}#1\hspace{1em}}}

\DeclareMathOperator*{\diag}{\mathbf{diag}}

\definecolor{myblue}{HTML}{D2E4FC}
\definecolor{Gray}{gray}{0.92}

\usepackage{thmtools}
\usepackage{thm-restate}

\usepackage{wrapfig}

\DeclareDocumentCommand{\Prto} {o} {
  \IfNoValueTF {#1}
  {\overset{\Pr}{\longrightarrow}}
  { \xrightarrow[ #1 \to \infty]{\Pr }}
}

\usepackage{microtype}
\usepackage{graphicx}
\usepackage{subfigure}
\usepackage{booktabs} % for professional tables
\usepackage{caption}

\usepackage{float}
\usepackage{wrapfig}
\usepackage{hyperref}

\usepackage{array}
\usepackage{arydshln}
\usepackage[accepted]{icml2022}

\usepackage{amsmath}
\usepackage{amssymb}
\usepackage{mathtools}
\usepackage{amsthm}

\usepackage[capitalize,noabbrev]{cleveref}

\usepackage[textsize=tiny]{todonotes}

\icmltitlerunning{Only Tails Matter:
Average-Case Universality and Robustness in the Convex Regime}

\begin{document}

\twocolumn[
\icmltitle{Only Tails Matter:
Average-Case Universality and\\ Robustness in the Convex Regime}



\icmlsetsymbol{equal}{*}

\begin{icmlauthorlist}
\icmlauthor{Leonardo Cunha}{mila}
\icmlauthor{Gauthier Gidel}{mila,cifar}
\icmlauthor{Fabian Pedregosa}{google}
\icmlauthor{Damien Scieur}{sait}
\icmlauthor{Courtney Paquette}{mcgill}
\end{icmlauthorlist}
\icmlaffiliation{cifar}{Canada CIFAR AI Chair}
\icmlaffiliation{google}{Google Research}
\icmlaffiliation{mila}{MILA and DIRO, Université de Montreal,
  Montreal, Canada}
\icmlaffiliation{mcgill}{McGill University, Montreal, Canada}
\icmlaffiliation{sait}{Samsung SAIT AI Lab, Montreal, Canada}
%\icmlaffiliation{sch}{School of ZZZ, Institute of WWW, Location, Country}

\icmlcorrespondingauthor{Leonardo Cunha}{leonardocunha2107@gmail.com}

\icmlkeywords{Machine Learning, ICML}

\vskip 0.3in
]

\printAffiliationsAndNotice{} 
\begin{abstract}
    The recently developed average-case analysis of optimization methods allows a more fine-grained and representative convergence analysis than usual worst-case results. In exchange, this analysis requires a more precise hypothesis over the data generating process, namely assuming knowledge of the expected spectral distribution (ESD) of the random matrix associated with the problem. This work shows that the concentration of eigenvalues near the edges of the ESD determines a problem's asymptotic average complexity. This a priori information on this concentration is a more grounded assumption than complete knowledge of the ESD. This approximate concentration is effectively a middle ground between the coarseness of the worst-case scenario convergence and the restrictive previous average-case analysis. We also introduce the Generalized Chebyshev method, asymptotically optimal under a hypothesis on this concentration and globally optimal when the ESD follows a Beta distribution. We compare its performance to classical optimization algorithms, such as gradient descent or Nesterov's scheme, and we show that, in the average-case context, Nesterov's method is universally nearly optimal asymptotically.
\end{abstract}
\section{Introduction}

The analysis of the average complexity of algorithms has a long tradition in computer science. Average-case complexity, for instance, drives much of the decisions made in cryptography \citep{bogdanov2006average}.

Despite their relevance, average-case analyses are difficult to extend to other algorithms, partly because of the intrinsic issue of defining a typical distribution over problem instances. Recently though, \citet{pedregosa2020acceleration}  derived a framework to systemically evaluate the complexity of first-order methods when applied to distributions of quadratic minimization problems. This derivation is done by relating the average-case convergence rate to the \textit{expected spectral distribution} (ESD) of the objective function's Hessian, which is a well-studied object in random matrix theory. In practice, however, the knowledge of the ESD is a much stronger requirement than the worst-case analysis, which relies only on knowledge of the distribution's support. 

\citet{paquette2020halting}  extended the average-case framework by introducing a noisy generative model for the problems. Among other results, they derived the average complexity of the Nesterov Accelerated Method \citep{nesterov2003introductory} on a particular distribution. Moreover, they showed the concentration of convergence metrics in the infinite sample and dimensional limit. 

\citet{scieur2020universal} showed that for a strongly convex problem with eigenvalues supported on a contiguous interval, the optimal average-case complexity converges asymptotically to the one given by the Polyak Heavy Ball method \citep{Polyak1962Some} in the worst-case.

\begin{figure}[H]
\centering
\includegraphics[clip, trim=4cm 0cm 0cm 0cm,width = 0.8 \linewidth]{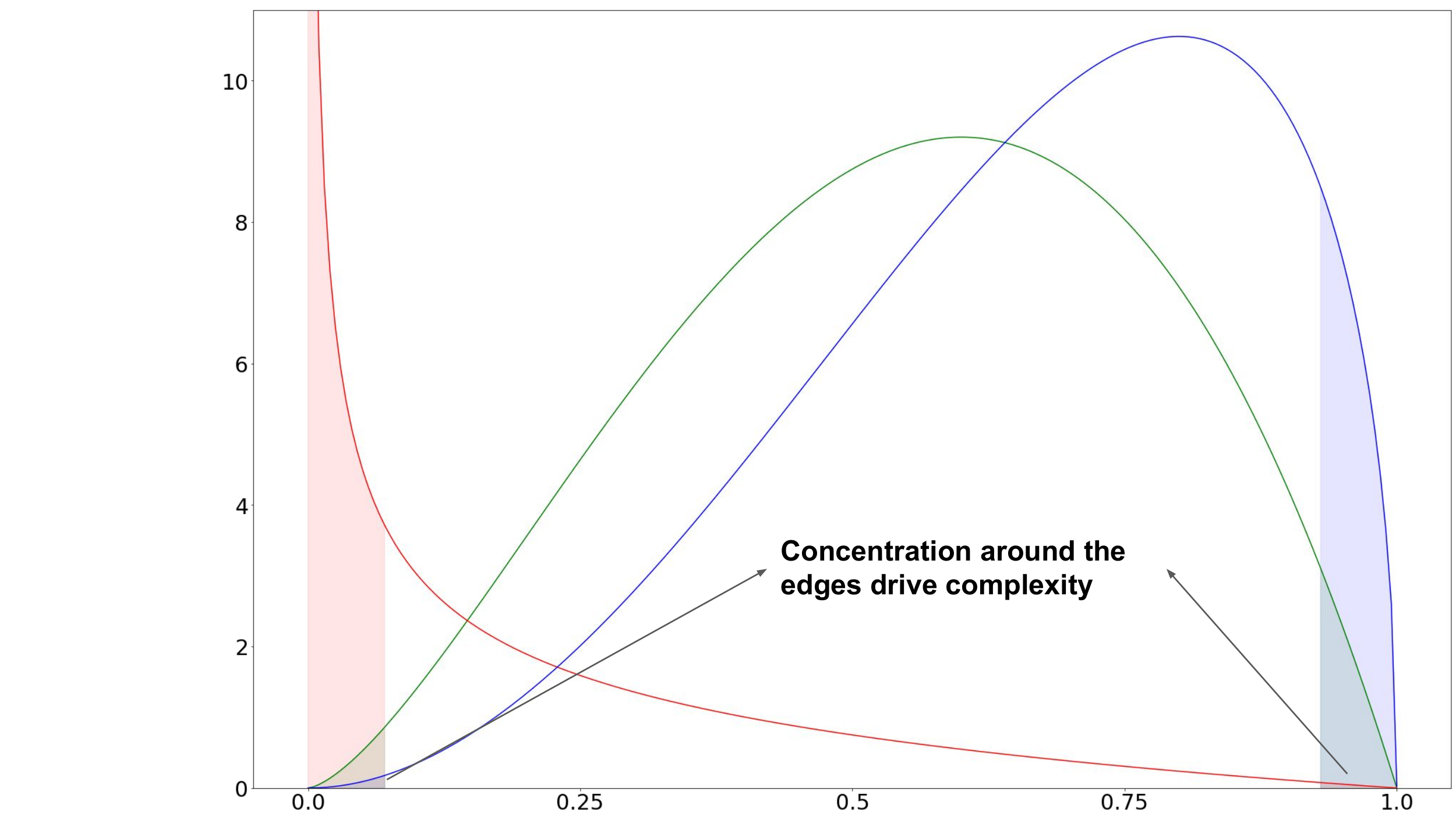}

%\begin{small}
\caption{
Representation of different spectra with different concentrations of eigenvalues around the support edges. These concentrations determine the average-case rates for non-strongly problems.}
%\end{small}
\end{figure}

\subsection{Current limitations of the average-case analysis} 
This paper addresses some of the limitations in previous works on average-case analysis. First, little is known about the convergence rate on \textbf{convex} but not strongly-convex problems. Also, optimal average-case algorithms require an \textbf{exact estimation of the ESD} to guarantee an optimal convergence rate; otherwise, their convergence rate under inexact ESD is unknown. Finally, the \textbf{non-smooth} has not yet been discussed in its full generality.

\textbf{Convex problems.}  In the smooth and strongly convex case, the optimal worst-case and average-case convergence rates are asymptotically equal \citep{scieur2020universal}. However, little is known about optimal average-case rates for non-strongly convex problems, as well as the average-case complexity of classical methods such as gradient descent or Nesterov's method, see \citep{paquette2020halting}.

\textbf{Exact estimation of the ESD.} \citet{pedregosa2020acceleration} developed average-case optimal algorithms that required an exact estimation of the ESD of the problem class. Such an estimation may be complex, if not impossible, to obtain in practice. Despite showing good performance when the ESD is estimated with empirical quantities, there are no theoretical guarantees of the method's performance when the ESD is poorly estimated. Therefore, there is a need to analyze the algorithm's performance under different notions of uncertainty on the spectrum, allowing a practitioner to choose the best algorithm for a practical problem even with imperfect \textit{a priori} information.

\textbf{Non-smooth objectives.} In this paper we provide an average-case analysis for the generalized Laguerre distribution $\lambda^\alpha e^{-\lambda},\, \alpha >-1$. This is a distribution with unbounded support, i.e., when the largest eigenvalue (also known as smoothness constant) is bounded. This extends the results of 
\citet{pedregosa2020acceleration} for the Laguerre distribution $e^{-\lambda}$.

\subsection{Contributions}

Our main contribution is a fine-grained analysis of the average-case complexity on convex quadratic problems: we show that a problem's complexity depends on the concentration of the eigenvalues of ESD around the edges of their support. Using this analysis, we also derive a family of average-case \textbf{optimal}  algorithms, analyze their \textbf{robustness}, and finally exhibit a \textbf{universality} result for Nesterov's method.
\begin{itemize}[leftmargin=*]
\setlength\itemsep{0.1 cm}
    \item \textbf{Optimal algorithms.} In Section \ref{section: methods}, we propose the Generalized Chebyshev Method (GCM, Algorithm \ref{algo: chebyshev}), a family of algorithms whose parameters depend on the concentration of the ESD around the edges of their support. With the proper set of parameters, the GCM method  converges at an optimal average-case rate (Theorem \ref{thm: optimality} for smooth problems, Theorem \ref{thm:laguerrerates} for non-smooth problems). These rates are faster than optimal worst-case methods like Nesterov acceleration, and we recover the classical worst-case rates as limits of the average-case (see Table~\ref{table: rates}). Fig.~\ref{fig: last figure} shows our theoretical analysis to match the practical performance of the algorithms.
    \item \textbf{Robustness.} Developing an optimal algorithm requires knowledge of the exact ESD. However, in practical scenarios, we only have access to an \textit{approximation} of the ESD. In Theorem \ref{thm: jacobirates} in Section \ref{section: robust average}, we analyze the rate of GCM in the presence of such a mismatch. We also analyze the optimal average-case rates of distributions representing the smooth convex, non-smooth convex, and strongly convex settings and compare them with the worst-case rates (Table~\ref{table: rates}).
    \item \textbf{Universality.} Finally, in Theorem \ref{the: neste rates}, we analyze the asymptotic average-case convergence rate of Nesterov's method. We show that its convergence rate is optimal up to a logarithmic factor under some natural assumptions over the data, namely a concentration of eigenvalues around $0$ similar to the Marchenko-Pastur distribution. This observation contributes to the theoretical understanding of the empirical effectiveness of Nesterov's acceleration.
\end{itemize}

\begin{figure}[H]
\centering
\captionsetup{type=table}
    \bgroup
    \def\arraystretch{1.3}
        \small
        \begin{tabular}{@{}lcc@{}}\toprule
            \textbf{Regime} & \textbf{Worst-case} & \textbf{Average-Case} \\
            \midrule
            Strongly conv. & $\left(1-\Theta(\nicefrac{1}{\sqrt{\kappa}})\right)^t$ &
            $\left(1-\Theta(\nicefrac{1}{\sqrt{\kappa}})\right)^t$\\ 
            \hdashline
            Smooth conv. & $\nicefrac{1}{t^2}$ & $\nicefrac{1}{t^{2\xi+4}}$\\
            \hdashline
            Convex & $\nicefrac{1}{\sqrt{t}}$ & $\nicefrac{1}{t^{\alpha+2}}$\\
            \bottomrule
        \end{tabular}
        \label{tab:theoretic rates}
    \egroup
    \vspace{ 0.5 cm}
    \caption{Comparison between function value worst-case and average-case convergence. $\kappa$ is the condition number in the smooth strongly convex case. In the smooth convex case $\xi> -1$ is the concentration of eigenvalues around $0$  and in the non-smooth case we consider $\dif\mu\propto \lambda^\alpha e^{-\lambda} d\lambda $.}\label{table: rates}

\end{figure}
\section{Average-Case Analysis} \label{section: average case}

This section recalls the average-case analysis framework for random quadratic problems.
The main result is Theorem~\ref{thm: metrics}, which relates the expected error  to the \textit{expected spectral distribution} and the \textit{residual polynomial}. The one-to-one correspondence between the residual polynomials and first-order methods applied to quadratics will allow us to pose the problem of finding an optimal method as the best approximation problem in the space of polynomials.

We  define a \textbf{random} quadratic problem as follows:
\begin{problem}
Let $\HH \in \RR^{d \times d}$ be a random symmetric positive-definite matrix independent of $\xx^\star \in \RR^d$, which is a random vector and a solution to the problem. We define the random quadratic minimization problem as
\begin{empheq}[box=\mybluebox]{equation*}\tag{OPT}\label{eq:quad_optim}
  \vphantom{\sum_0^i}\min_{\xx \in \RR^d} \Big\{ f(\xx) \defas\!\mfrac{1}{2}(\xx\!-\!\xx^\star)^\top\!\HH(\xx\!-\!\xx^\star) \Big\}\,.
\end{empheq}
We are interested on minimizing the expected errors $\EE \|f(\xx_t) - f(\xx^\star)\|$, the expected function-value gap, and $\EE \|\nabla f(\xx_t)\|^2$, the expected gradient norm, where $\xx_t$ is the $t$-th update of a first-order method starting from $\xx_0$ and $\EE$ is the expectation over the random variables $\HH, \xx_0$ and $\xx^\star$. 
\end{problem}

Note that the expectations we consider throughout the paper are over problem instances, not over any randomness of the algorithm.

In this paper, we consider the class of \emph{first-order methods} (F.O.M's) to minimize \eqref{eq:quad_optim}. Methods in this class construct the iterates $\xx_t$ as
\begin{equation} \label{eq:first_order_methods}
    \xx_{t} \in \xx_0 + \Span\{ \nabla f(\xx_0), \ldots, \nabla f(\xx_{t-1})  \}\,,
\end{equation}
that is, $\xx_t$ belongs to the span of previous gradients. This class of algorithms includes for instance gradient descent and momentum, but not quasi-Newton methods since the preconditioner could allow the iterates to go outside of the span. Furthermore, we will only consider \emph{oblivious} methods, that is, methods in which the coefficients of the update are known in advance and do not depend on previous updates. This leaves out some methods such as conjugate gradient or methods with line-search.

\paragraph{From First-Order Methods to Polynomials.}
There is an intimate link between first-order methods and polynomials that simplifies the analysis of quadratic objectives. The following proposition shows that, with this link, we can assign to each optimization method a polynomial that determines its convergence.
% The following proposition states this relationship which relates the error at iteration $t$ with the error at initialization and the residual polynomial.
We will denote $P_t(\lambda)$ and $P_t(\HH)$ for the polynomials taking arguments in $\RR$ and $\RR^{n\times n}$ respectively with the same coefficients. 
Following \cite{fischer1996polynomial}, we will say a polynomial $P_t$ is \textit{residual} if $P_t(0)=1$.\\

\begin{restatable}{proposition}{linkalgopolynomial}
    \label{prop:link_algo_polynomial} \citep{hestenes1952methods}
    Let $\xx_t$ be generated by a first-order method. Then there exists a residual polynomial $P_t$ of degree $t$, that verifies
    \begin{equation}\label{eq:polynomial_iterates}
        \vphantom{\sum^n}\xx_{t}-\xx^\star = P_t(\HH)(\xx_0-\xx^\star)~.
    \end{equation}
\end{restatable}

\begin{remark} \label{rmk: momentum based}
If the first-order method  is further a \textbf{momentum method}, i.e.
$$
    \xx_{t+1}=\xx_t+h_t\nabla f(\xx_t)+m_t(\xx_t-\xx_{t-1}) ,
$$
we can determine the polynomials by the recurrence $P_0=1$ and
    \begin{equation*}
        P_{t+1}(\lambda)=P_t(\lambda)+h_t\lambda P_t(\lambda)+m_t(P_t(\lambda)-P_{t-1}(\lambda)) .
    \end{equation*}
We note that while most popular F.O.M's can be posed as a momentum method, the Nesterov method cannot.
\end{remark}

A convenient way to collect statistics on the spectrum of a matrix is through its \emph{empirical spectral distribution}.\\

\begin{definition}
[Expected spectral distribution (ESD)]. 
Let $\HH$ be a random matrix with eigenvalues $\{\lambda_1, \ldots, \lambda_d\}$. The \textbf{empirical spectral distribution} of $\HH$, called ${\mu}_{\HH}$, is the probability measure
\begin{equation}\label{eq:wighted_spectral_density}
    \mu_{\HH} \defas \frac{1}{d}{\textstyle{\sum_{i=1}^d}} \delta_{\lambda_i},
\end{equation}
where $\delta_{\lambda_i}$ is the Dirac delta, a distribution equal to zero everywhere except at $\lambda_i$ and whose integral over the entire real line is equal to one.

Since $\HH$ is random, the empirical spectral distribution $\mu_\HH$ is a  random variable in the space of measures. Its expectation over $\HH$ is called the \textbf{expected spectral distribution} (ESD) and we denote it
\begin{equation}
\mu \defas \EE_{\HH}[\mu_{\HH}]\,.
\end{equation}
\end{definition}

The following theorem links the ESD to the average-case convergence of a first-order method when $\xx_0-\xx^\star$ and $\HH$ are independent.

\begin{restatable}{theorem}{metrics} \label{thm: metrics}
Let $\xx_t$ be generated by a first-order method with associated residual polynomial $P_t$, $\mu$ be the ESD, and $\mathbb{E}[(\xx_0-\xx^\star)(\xx_0-\xx^\star)^\top]=R^2\textbf{I}$ for some constant $R$. Then we have the following identities for different convergence metrics:
\label{eq:error_norm_x}
\begin{align}
  &\mathbb{E}[\|\xx_t-\xx^\star\|^2] = { R^2} \int {P_t^2(\lambda) \dif\mu(\lambda)}\,,\\
    &\mathbb{E}[f(\xx_t)-f(\xx^\star)]=\frac{R^2}{2}\int P_t^2(\lambda)\lambda \dif\mu(\lambda)\,,\\
    &\mathbb{E}[\|\nabla f(\xx_t)\|^2_2]=R^2\int P_t^2(\lambda)\lambda^2\dif\mu(\lambda)\,.
\end{align}
\end{restatable}
This theorem states that polynomials are a powerful abstraction, allowing us to write all our convergence metrics within the same framework. Therefore, in the rest of the paper, we will refer directly to the polynomials associated with a given method. For simplicity, we set $R^2 = 1$.

This framework is linked to the field of \textbf{orthogonal polynomials} by the following proposition. We construct an optimal method w.r.t. a given distribution through a family of orthogonal polynomials.

\begin{restatable}{proposition}{optimality}
 \label{prop: optimality}
 Let $\nu$ be a distribution with continuous support and $P_t^l$ be defined as
 \begin{equation}
     P_t^l\defas{\arg \min}_{P_t(0)=1} \int P_t^2(\lambda) \lambda^l d\nu(\lambda)\,.
 \end{equation}
 Then $(P_t^l)_{t=0,1,\ldots}$ is the family of residual orthogonal polynomials w.r.t. to $\lambda^{l+1}d\nu$.
\end{restatable}

We refer to objective $l$ as the one associated to the added $\lambda^l$ term, i.e. the function-value is objective $l=1$. This proposition further implies that the optimal first-order method is a momentum method because Favard's theorem \cite{marcellan2001favard} tells us the orthogonal polynomials w.r.t. a given distribution are related through a \textbf{three term recurrence},

\begin{equation}
    P_{t+1}(\lambda)=(a_t + b_t\lambda) P_t(\lambda)+(1-a_t)P_{t-1}(\lambda)\,.
\end{equation}
Following Remark \ref{rmk: momentum based}, the optimal method is derived from this recurrence as
\begin{equation}
    \xx_{t+1}=\xx_t+(a_t-1)(\xx_t-\xx_{t-1})+b_t\nabla f(\xx_t)\,.
\end{equation}

\section{Methods} \label{section: methods}
Writing the rates in terms of the \textit{expected spectral distribution} ties the average-case framework to the field of \textit{random matrix theory}. A classical result of this field is that the same spectral distribution arises from different input distributions in the data. For example, the \textbf{Marchenko-Pastur} distribution arises in the large sample and dimension limit when we take the Gram matrix of a matrix where the entries are generated i.i.d. from any distribution with mean zero, variance $\sigma^2$ and with a bounded fourth moment.

\begin{definition}
The  Marchenko-Pastur distribution associated with the parameter  $r$ and with scale $\sigma^2$ is given by 
\begin{equation}
\dif\mu_{MP}(\lambda)=\frac{1}{2\pi\sigma^2}\frac{\sqrt{(\lambda^+-\lambda)(\lambda-\lambda^-)}}{r\lambda}\dif\lambda \,,
\end{equation}
with $\lambda^+=\sigma^2(1+\sqrt{r})^2$, $\lambda^-=\sigma^2\max(0,(1-\sqrt{r})^2)$.
\end{definition}
 $r$  in the Gram matrix case, is determined as $n/d$. The Marchenko-Pastur distribution $\mu_{MP}$ can be considered a natural first model for e.s.d's as it arises universally from matrices with i.i.d. entries,under mild low moment assumptions, there is no specific distribution of the matrix to be considered. It can be seen as a model for the  white-noise in the data. When $r=1$, i.e. $n=d$, we have $\dif\mu_{MP}\propto \lambda^{-1/2}\sqrt{\lambda^+-\lambda} \dif \lambda$. %Though practical e.s.d's do not take the exact shape of the MP distribution, the same concentration near $0$ is often verified.

\citet{pedregosa2020acceleration} first derived the optimal method w.r.t. $\mu_{MP}$, and \citet{paquette2020halting} derived Nesterov's rates under the distribution. In this paper we take a more general view and consider Beta distribution.
\begin{definition}
    The (generalized) Beta distribution with parameters $\tau,\xi$ and scale $L$ are given by the (non-normalized) probability density function
    \begin{equation}
        \dif\mu_{\tau,\xi}(\lambda) \propto \lambda^\xi(L-\lambda)^\tau\dif\lambda\,,
    \end{equation}
\end{definition} 
where $\dif \lambda$ denotes the standard Lebesgues measure on $\mathbf{R}$.
This family of distributions generalizes the Marchenko-Pastur distribution, and both have similar concentrations near $0$ when $\xi\approx -1/2$.

The optimal method w.r.t. $\mu_{\tau,\xi}$ and objective $l$ is associated to a shifted Jacobi polynomial $\tp_t^{\alpha,\beta}$ with $\beta=\xi+l+1, \alpha=\tau$. This is a direct implication of the definition of the Jacobi polynomials as the orthogonal family w.r.t. the weights $(1-\lambda)^\alpha(1+\lambda)^\beta$. When $\alpha=\beta=-1/2$, we retrieve the \textit{Chebyshev Method} \citep{flanders1950numerical}. 
It is also related to the ``$\nu$-method'' of \citet{brakhage1987ill}, who considers the conjugate gradient algorithm applied to a matrix with a Gamma spectral distribution ($\beta=-1/2$).

We name the following method the \textit{Generalized Chebyshev Method} (GCM). 

\begin{algorithm}
\caption{Generalized Chebyshev Method GCM($\alpha,\beta$)}
\begin{algorithmic}
   
\STATE\textbf{Inputs}: Initial vector $\xx_0$, distribution parameters $\alpha, \beta$, $L = $largest eigenvalue of $\HH$\,.   \\

$\xx_{-1}\gets \xx_0,\delta_0\gets0$   \\
\FOR{$t=1,\ldots,T$}{
\STATE $a_t\gets-\frac{2\left(\beta^2+\alpha\beta+(2t+1)(\alpha+\beta)+2t^2+2\right)(2t+\alpha+\beta+1)}{
            2(t+1)(t+\alpha+\beta+1)(2t+\alpha+\beta)} $\\
\STATE $b_t\gets\frac{(2t+\alpha+\beta+1)(2t+\alpha+\beta+2)}{
            L(t+1)(t+\alpha+\beta+1)}$ \\
\STATE $\gamma_t\gets-\frac{(t+\alpha)(t+\beta)(2t+\alpha+\beta+2)}{
            (t+1)(t+\alpha+\beta+1)(2t+\alpha+\beta)}$\\
\STATE $\delta_t\gets\frac{1}{a_t+\gamma_t\delta_{t-1}}$\\
\STATE $\xx_t\gets\xx_{t-1}+(\delta_ta_t-1)(\xx_{t-1}-\xx_{t-2})+\delta_t b_t\nabla f(\xx_{t-1})$
}
\ENDFOR
\end{algorithmic}
\label{algo: chebyshev}
\end{algorithm}

We consider the Nesterov's method used in \cite{paquette2020halting}, which is defined by the iterations:
\begin{align}
    \xx_{t+1}&=\yy_t-\frac{1}{L}\nabla f(\yy_t), \\
    \yy_{t+1}&=\xx_{t+1}+\frac{t}{t+3}(\xx_{t+1}-\xx_t)\,.
\end{align}
We  also consider the Laguerre method \cite{pedregosa2020acceleration}, which is optimal w.r.t. $\dif\mu(\lambda)=\frac{\lambda^\alpha e^{-\lambda}}{\Gamma(\alpha+1)}\dif \lambda$, taking $\alpha$ as a parameter. This method is proposed to optimize non-smooth functions.

Both these methods are generalizations of the ones in \cite{pedregosa2020acceleration}. We show that Algorithm \ref{algo: chebyshev} corresponded to polynomials $\tp_t^{\alpha,\beta}$ and derive the Laguerre method in Appendix \ref{jacobi recurrence}.

\begin{remark}
The Generalized Chebyshev takes the largest eigenvalue $L$ as a parameter, but the rates we will show are robust to an \textit{overestimation} of $L$.
\end{remark}

\section{Robust Average-Case Rates} \label{section: robust average}
Throughout the rest of the paper, we make the following assumption about the spectral distributions:
\begin{assumption}
Let $\nu_{\tau,\xi}$ be a distribution supported in $(0,L]$ s.t. $\nu_{\tau,\xi}'(x)>0$ for $x\in [0,L]$, $d\nu_{\tau,\xi}=\Theta( \lambda^\xi)$ near $0$ and $d\nu_{\tau,\xi}=\Theta( (L-\lambda)^\tau)$ near $L$.
\end{assumption} 
We characterize our distributions of interest only in terms of the behavior at the edges. This assumption is sufficient to determine the asymptotic convergence of algorithms. This (mild) assumption excludes only distributions that decay exponentially near their edges, in which case they effectively behave as strongly convex problems. Moreover, this assumption allows considering a broader class of problems than previous work that assumed complete knowledge of the spectrum of the Hessian.

The $\xi$ parameter measures how close we are to the worst-case scenario as it approaches $-1$. Samples in finite dimension of distributions with high values of $\xi$ will work as strongly convex functions in practice, as samples of low eigenvalues are rarer and rarer with increasing $\xi$.

We show that the coefficients $\xi,\,\tau$ determine the asymptotics of the convergence of the methods: only the concentrations near the edge matter. We do this by singling out from each of these classes the beta distributions for which we can compute the rates, then show the rates to be the same for all distributions with the same concentrations.

\begin{restatable}[GCM average-case rates]{theorem}{robustjacobi}\label{thm: jacobirates}
The Generalized Chebyshev Method with parameters $(\alpha,\beta)$ applied to a problem with expected spectral distribution $\nu_{\tau,\xi}$ has average-case rates ${\mathbb{E}[f(\xx_t)-f(\xx^\star)]}\sim L\cdot C^{\alpha,\beta}_{1,\nu} t^{e_1}$ and  $\mathbb{E}[\|\nabla f(\xx_t)\|^2_2]\sim L^2\cdot C^{\alpha,\beta}_{2,\nu}t^{e_2} $, with the exponents 
\begin{align*}
e_1=&\left\{\begin{array}{ll}
    -1-2\beta \hspace{.3 cm}\mbox{if} \hspace{.3 cm}
		  \alpha<\tau+\tfrac12 \;\; \land \;\;  \beta <\xi+\tfrac32,&\\
		  -2(\xi+2)\log t \hspace{.3 cm}\mbox{if}\hspace{.3 cm} 
		  \alpha=\tau+\tfrac12 \;\;\land \;\; \beta =\xi+\tfrac32,&\\
		  2(\max\{\alpha-\beta-\tau,-\xi-1\}-1) \hspace{.3 cm}\mbox{otherwise},&
	\end{array}\right. \\
e_2=&\left\{\begin{array}{ll}
    -1-2\beta \hspace{.3 cm}\mbox{if} \hspace{.3 cm}
		  \alpha<\tau+\tfrac12 \;\;\land \;\; \beta <\xi+\tfrac52,&\\
		  -2(\xi+3)\log t \hspace{.3 cm}\mbox{if}\hspace{.3 cm} 
		  \alpha=\tau+\tfrac12 \;\;\land \;\; \beta =\xi+\tfrac52,&\\
		  2(\max\{\alpha-\beta-\tau,-\xi-2\}-1) \hspace{.3 cm}\mbox{otherwise}\,,&
	\end{array}\right.
\end{align*}
where $\land$ denotes the logical ``\emph{and}'' operator and $C^{\alpha,\beta}_{.,\nu}$ are distribution dependent constants.
\end{restatable}

\begin{sketch}
We first compute the convergence rates assuming that the ESD is a generalized Beta distribution with parameters $\tau, \xi$. For this we use Theorem \ref{thm: metrics} and  Lemma \ref{assymp. lemma}. For a general $\nu_{\tau,\xi}$, with the aid of Lemma \ref{wk}, we then show that the mass away from the edges is negligible, i.e.,
\begin{equation}
    \int_\epsilon^{L-\epsilon}P_t^{\alpha,\beta}(\lambda)^2\lambda^ld\nu_{\tau,\xi}\,,
\end{equation}
is $O(t^{-1-2\beta})=O(t^{e_.}), \forall \epsilon >0$. As the mass near the edges must be similar to the one for the Beta weights, the result follows.
\end{sketch}

Theorem \ref{thm: jacobirates}, illustrated by Fig. 2, shows that overestimating $\beta$ and underestimating $\alpha$ will still leave us with the optimal asymptotic rates. So a good rule of thumb for calibrating the algorithm is to use high $\beta$ and low $\alpha$.

Theorem \ref{thm: optimality} shows that a proper choice of $\alpha,\beta$ can indeed make the Jacobi polynomial asymptotically optimal w.r.t. to any $\nu_{\tau,\xi}$. 

\begin{restatable}[Optimal Rates]{theorem}{jacoptimal}\label{thm: optimality}
Consider the distribution $\nu_{\tau,\xi}$.
The optimal asymptotic average-case rates for $\mathbb{E}[f(\xx_t)-f(\xx^\star)]$ and $\mathbb{E}[\|\nabla f(\xx_t)\|^2_2]$ are attained by the GCM with parameters $(\tau,\xi+2)$ and $(\tau,\xi+3)$, respectively, and read
\begin{align}
    \mathbb{E}[f(\xx_t)-f(\xx^\star)] = \Theta(t^{-2(\xi+2)})\,, \\
    \mathbb{E}[\|\nabla f(\xx_t)\|^2_2] = \Theta(t^{-2(\xi+3)})\,.
\end{align}

\end{restatable}

\begin{sketch}
When $\nu_{\tau,\xi}$ is a Beta distribution, the result follows from Theorem \ref{thm: jacobirates}. 

For a general $\nu_{\tau,\xi}$, we consider the optimal method for this expected spectral distribution. We show that its performance on the corresponding Beta weight $\mu_{\tau,\xi}$ is similar to that of $\nu_{\tau,\xi}$, as the integrals concentrate on the edges. It follows that the optimal performance for $\nu_{\tau,\xi}$ is asymptotically the same as for $\mu_{\tau,\xi}$.
\end{sketch}
\begin{figure}[t]
    \centering
    \includegraphics[width=0.65\linewidth]{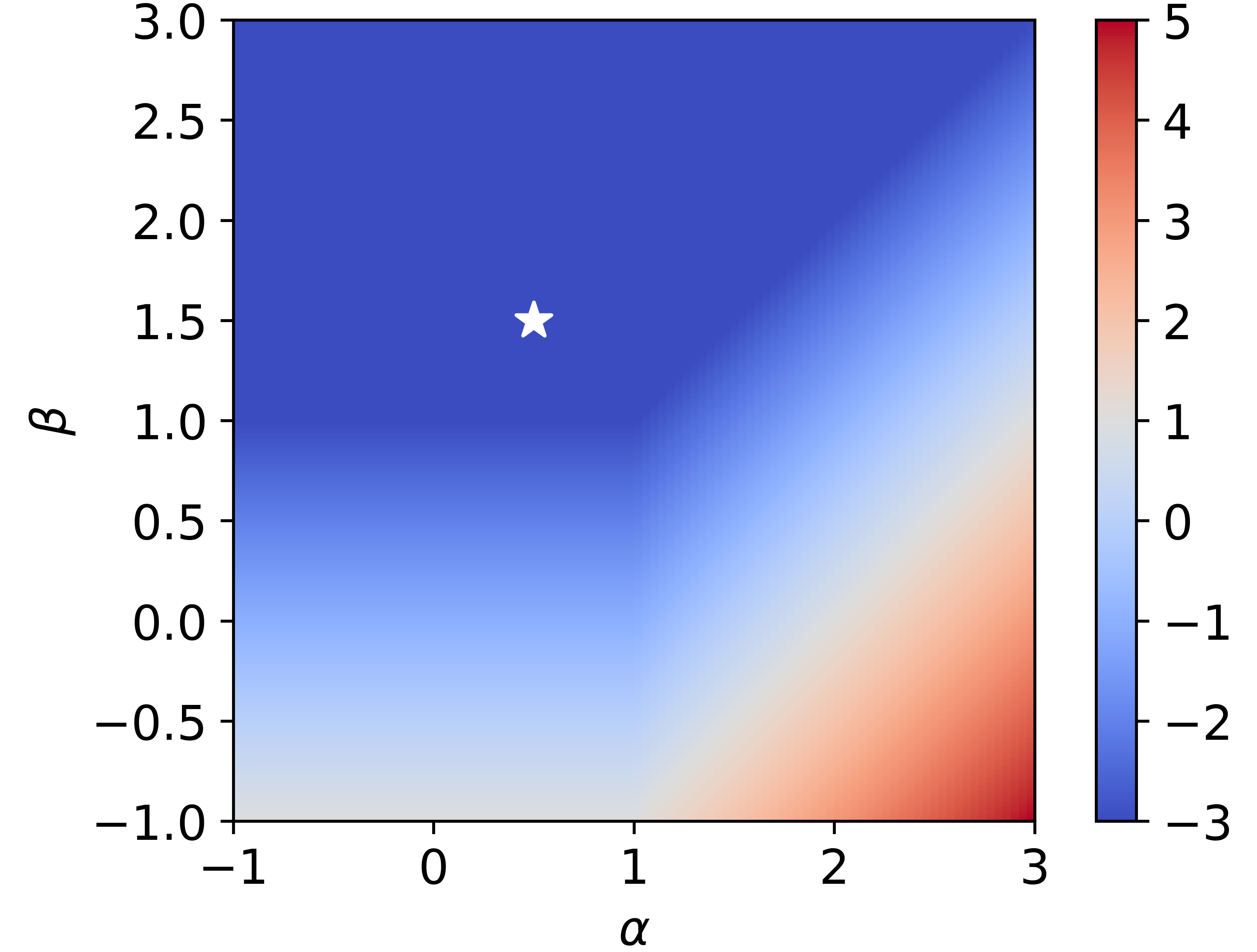}
    \caption{ The figure illustrates  the robustness of the Generalized Chebyshev Method with parameters $(\alpha,\beta)$ for a \emph{fixed problem} corresponding to the Marchenko-Pastur distribution ${(\tau=\tfrac12,\xi=-\tfrac12)}$. The 
    color represents the exponent $a$ of the average-case rate $O(t^{a})$ of the method for different values of $\alpha$ and $\beta$. The white star represents the optimal tuning and the blue area is the set of parameters for which the method converges. Note that a large region guarantees the same optimal asymptotic rate.}
\end{figure}
\begin{table}
    \centering

        \begin{tabular}{@{}lcc@{}}\toprule
            \multirow{2}{*}{\textbf{Method}$\qquad$} & \multicolumn{2}{c}{\textbf{Parameters} $(\tau,\xi)$} \\
            \cmidrule(lr){2-3}
             & $(\frac{1}{2},\frac{1}{2})$ & $(\frac{1}{2},-\frac{1}{2})$\\
            \cmidrule(lr){1-3}
            GCM$(\alpha=\frac{1}{2},\beta=\frac{5}{2})$ & $t^{-5}$ & $t^{-3}$ \\\hdashline
            GCM$(\alpha=\frac{1}{2},\beta=\frac{3}{2})$ & $t^{-4}$ & $t^{-3}$ \\\hdashline
            Nesterov & $t^{-4}$ &$t^{-3}\log t$\\\hdashline
            Gradient descent & $t^\frac{-5}{2}$ & $t^\frac{-3}{2}$ \\
            \bottomrule
        \end{tabular}
    % \end{table}
    \caption{ The table compares the  asymptotic average-case rates for the function-value for different methods and pairs $(\tau,\xi)$.}   \label{tab:fom rates}
  \end{table}

For the function value ($l=1$), we find rates that approach $t^{-2}$ as $\xi\rightarrow -1$, showing the worst-case as a limit (over the considered distribution) on the average-case.

We remark that the above theorems imply that, at least asymptotically, the GCM  is robust for a suboptimal choice of parameter $\beta$ up to $1/2$ below the optimal choice and infinitely above. 

For completeness, we derive worst-case rates for the GCM.
\begin{restatable}[GCM worst-case rates]{proposition}{worstcase}
Let $f$ be a convex, L-smooth quadratic function. Then, for the Generalized Chebyshev Method with parameters $(\alpha,\beta)$, we have worst-case rates $f(\xx_t)-f(\xx^\star) \leq C_1Lt^{e_3}$ and  $\|\nabla f(\xx_t)-f(\xx^\star)\|\leq C_2L^2 t^{e_4}$, with the exponents
\begin{align}
e_3=&\left\{\begin{array}{ll}
    2(\alpha-\beta) \hspace{.3 cm}\mbox{if} \hspace{.3 cm}
		  \alpha>\beta-1,&\\
		  -1-2\beta \hspace{.3 cm}\mbox{if}\hspace{.3 cm} 
		  \alpha\leq\beta-1 \land \beta \leq 1/2,&\\
		  -2\hspace{.3 cm}\mbox{otherwise},&
	\end{array}\right. \\
e_4=&\left\{\begin{array}{ll}
    2(\alpha-\beta) \hspace{.3 cm}\mbox{if} \hspace{.3 cm}
		  \alpha>\beta-2,&\\
		  -1-2\beta \hspace{.3 cm}\mbox{if}\hspace{.3 cm} 
		  \alpha\leq\beta-2 \land \beta \leq 3/2,&\\
		  -4\hspace{.3 cm}\mbox{otherwise}.&
	\end{array}\right.
 \end{align}
% Here the $\min$ operation is denoted by $\land$.
The logical and is noted $\land$.
\end{restatable}

For common choice of $\alpha,\beta$, i.e. $\beta\geq \frac{1}{2}$, $\alpha\leq \beta-1$, the rate of decay of the function suboptimality achieves the theoretical lower bound of $t^{-2}$.

We now analyze the convergence of the Nesterov method. \cite{nesterov2003introductory} has shown that it matches up to a  constant factor a lower bound on the worst-case complexity of non strongly convex problems. A natural question is if this performance would translate to good average-case rates. To do so, we extend the proof of  \citet[
Lemma B.2]{paquette2020halting} for the Nesterov method under the Marchenko-Pastur distribution.  
\begin{restatable}[Nesterov average-case rates]{theorem}{nesterovrates} \label{the: neste rates}
Consider the distribution $\nu_{\tau,\xi}$. Then for the Nesterov method, we have average-case rates
\begin{align}
    &\mathbb{E}[f(\xx_t)-f(\xx^\star)]\sim C'_{1,\nu}
    {\Bigg\{}\begin{array}{ll}
		  t^{-2(\xi+2)}& \mbox{if } 
		  \xi<-1/2,\\
		  t^{-3}\log t& \mbox{if } 
		  \xi=-1/2,\\
		  t^{-(\xi+7/2)}& \mbox{if } 
		  \xi>-1/2,
	\end{array} \\
	&\mathbb{E}[\|\nabla f(\xx_t)\|^2_2] \sim C'_{2,\nu}
		  t^{-(\xi+9/2)}.
\end{align}
\end{restatable}

The gap between the asymptotic average-case rates of Nesterov and the optimal ones is of the order of $t^{\xi+l-1/2}$, when $\xi+l>1/2$, $\log t$ when $\xi+l=1/2$ and $0$ otherwise.
This result shows that Nesterov is almost optimal when the concentrations near $0$ are relatively high, i.e., low $\xi$.

\begin{restatable}[Gradient descent average-case rates]{theorem}{gdrates} \label{the: gd rates}
Consider the distribution $\nu_{\tau,\xi}$. Then for gradient descent
\begin{align}
&\mathbb{E}[f(\xx_t)-f(\xx^\star)]=\Theta(t^{-(\xi+2)}),\\
	&\mathbb{E}[\|\nabla f(\xx_t)\|^2_2] =\Theta(t^{-(\xi+3)}).
\end{align}

\end{restatable}
\begin{sketch}
From the simple form of the polynomial associated with gradient descent, we have a closed-form (Eq. \eqref{eq: closed form}) expression for the performance of gradient descent on the Beta weights. For a general $\nu_{\tau,\xi}$, the mass that is $\epsilon$ away from $0$ is $O((1-\epsilon)^{2t})$; thus, the performance of gradient descent is asymptotically the same as on the Beta weights.
\end{sketch}

Observe for the function value that the rate for Nesterov is $t^{-2}$ and the rate for gradient descent is $t^{-1}$ when $\xi \rightarrow -1$.

We now consider the optimal rates for a Gamma distribution.

\begin{restatable}[Laguerre method rates]{theorem}{laguerrerates} \label{thm:laguerrerates}
Let $\alpha>-1$ and $\mu_\alpha$ be a Gamma distribution, i.e. $\dif\mu_\alpha(\lambda)={\lambda^\alpha e^{-\lambda}}/{\Gamma(\alpha+1)}\dif \lambda$. The optimal rates are given by the Laguerre method of appropriate tuning and
\begin{equation}
    \mathbb{E}[f(\xx_t)-f(\xx^\star)]=\Theta(t^{-(\alpha+2)})\,. 
    %%we can get the constants here
\end{equation}
\end{restatable}
Note that this result does not have the same universality as the others because of the non-compacity of the distribution's support. These rates are contrasted  to the worst-case lower bound on the optimization of non-smooth functions by first-order methods, which gives
\begin{equation*}
    f(\xx_k)-f(\xx^\star)\geq\frac{C}{\sqrt{t}}\,.
\end{equation*}
These rates are not found when $\alpha\rightarrow-1$, indicating that the worst-case is particularly pessimistic in this scenario.
\begin{remark}
All of the expected rates we state are almost deterministic on the high dimensional setting as per the concentration results shown in \cite{paquette2020halting}
\end{remark}

\section{Experiments}

\begin{figure*}
    \centering
    \includegraphics[width=.48\textwidth]{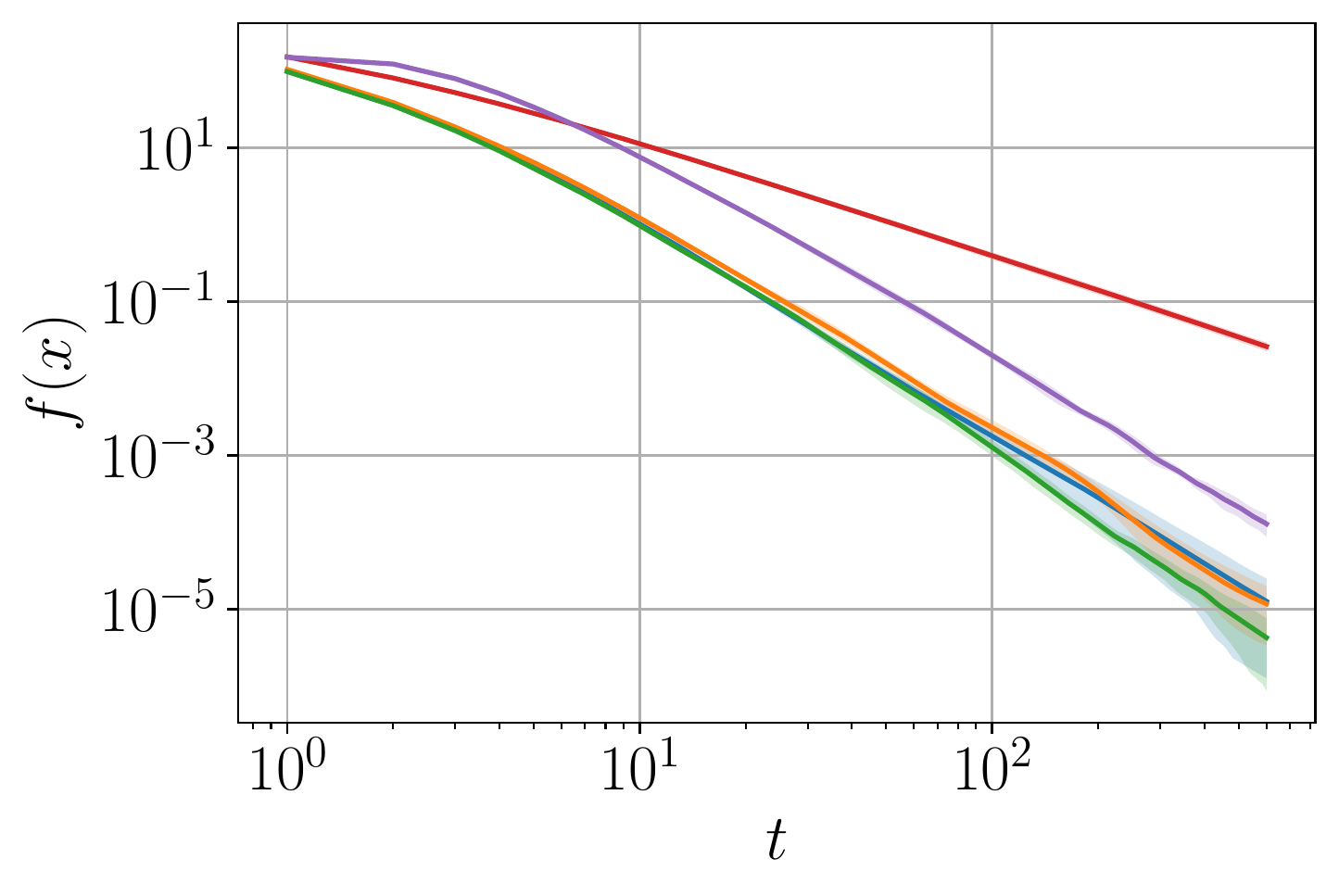}
    \hfill
    \includegraphics[width=.48\textwidth]{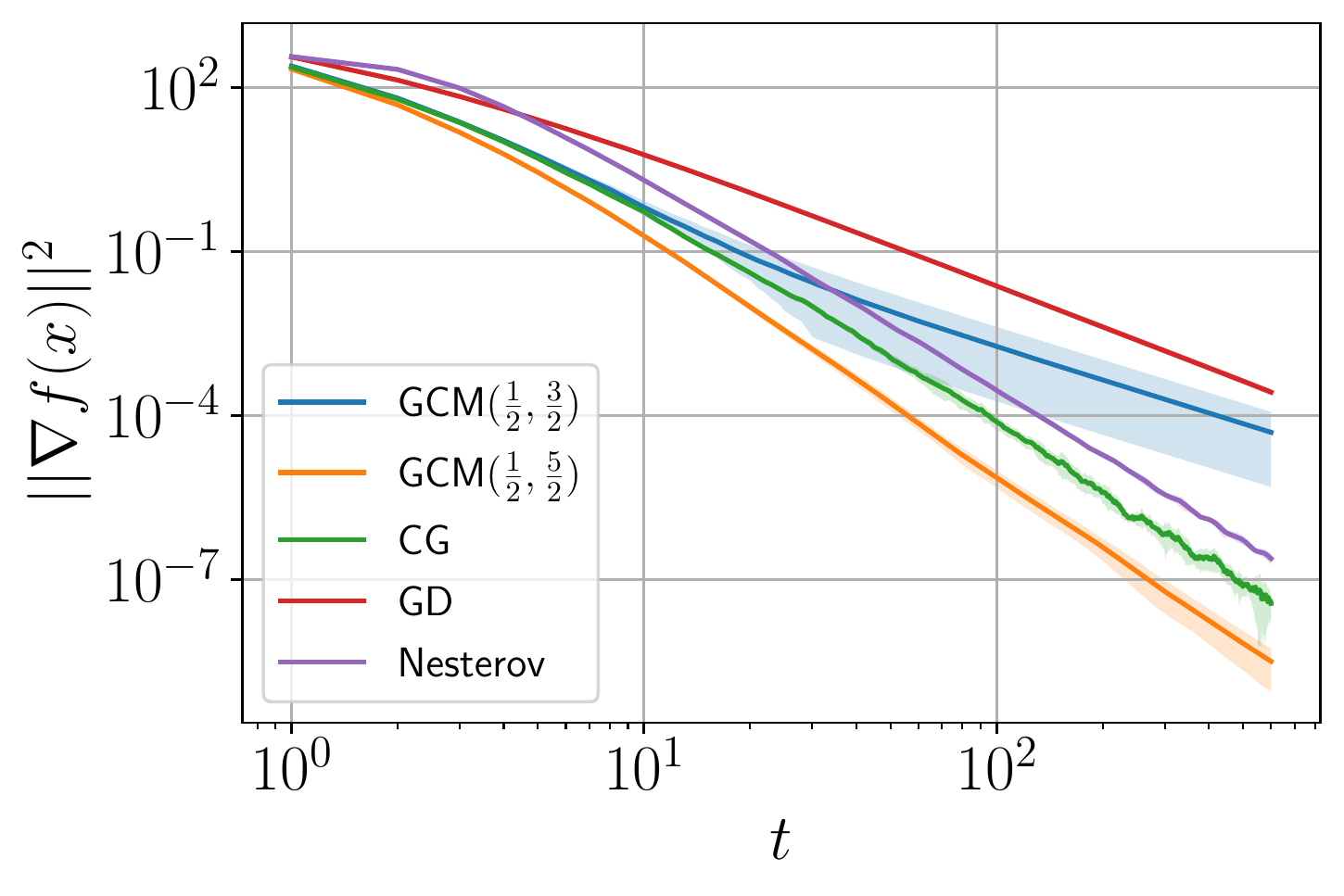}
    \caption{Rates for a synthetic problem, simulating the Marchenko-Pastur distribution. CG stands for conjugate gradient and GD for gradient descent. Shades are standard deviation for 8 different samplings of the distribution and of the random initialization. Note that both tunings of the GCM achieve performance in function value very close to the one of Conjugate Gradient, which is optimal for every draw of the problem.}
 \label{fig: mp performance}
    \vspace{1cm}
    \includegraphics[width=.48\textwidth]{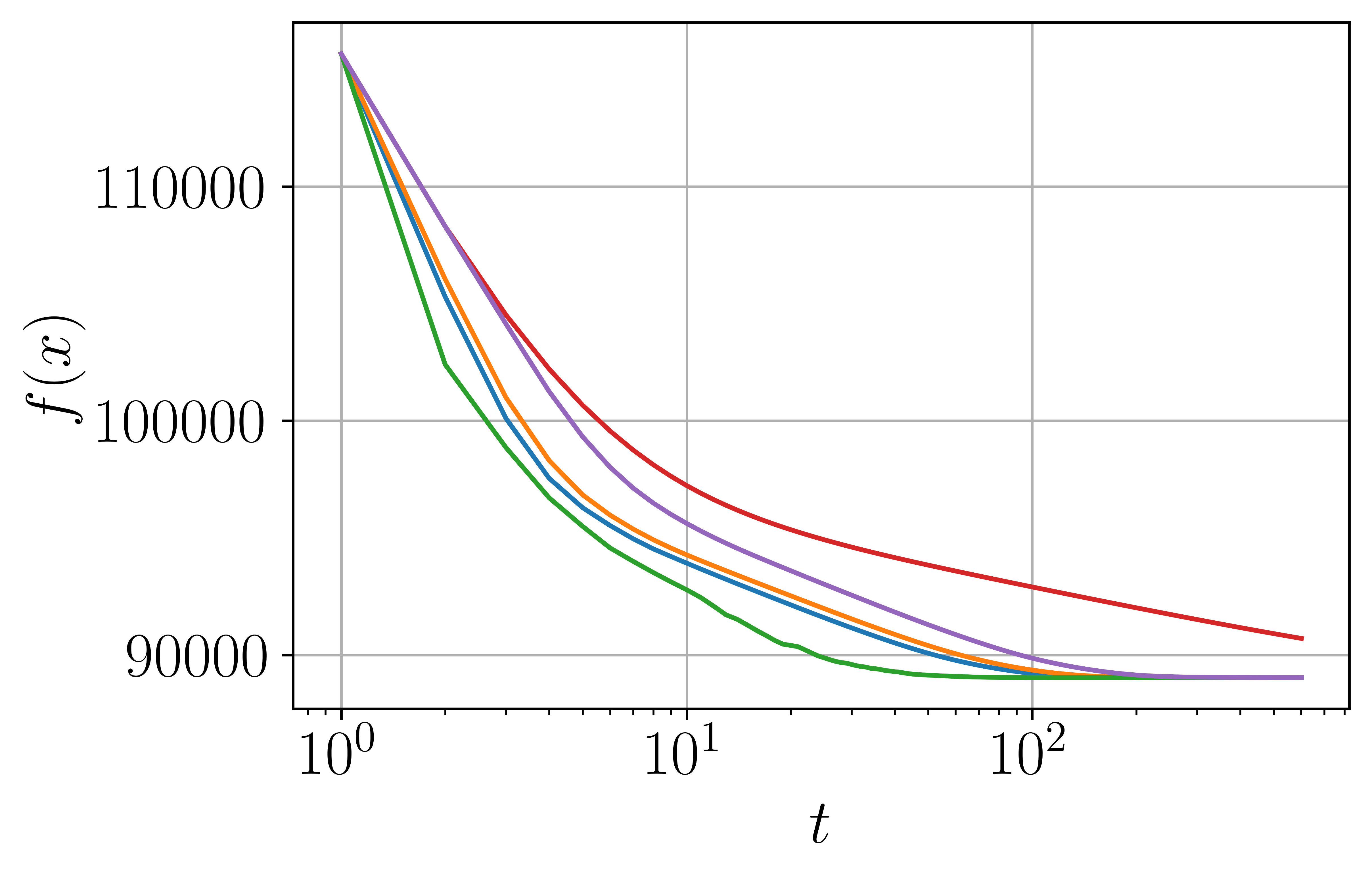}
    \hfill
    \includegraphics[width=.48\textwidth]{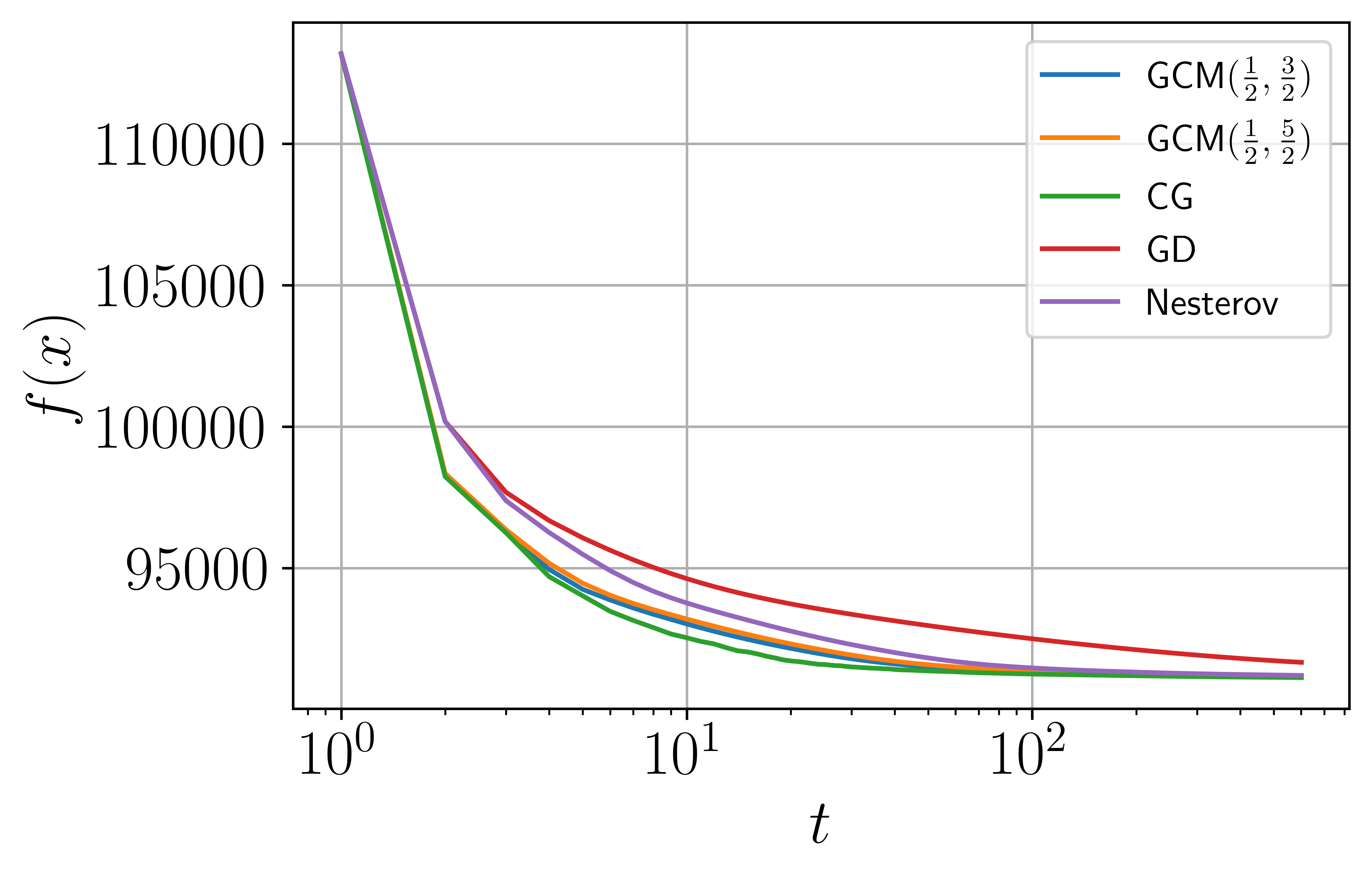}
    \caption{Function values for different methods where the data comes from CIFAR-10 Inception features \textit{(left)} and MNIST features \textit{(right)}. The properly tuned GCM achieves remarkable performance under these non-synthetic spectra.
}
    \vspace{1cm}
    \includegraphics[width=.48\textwidth]{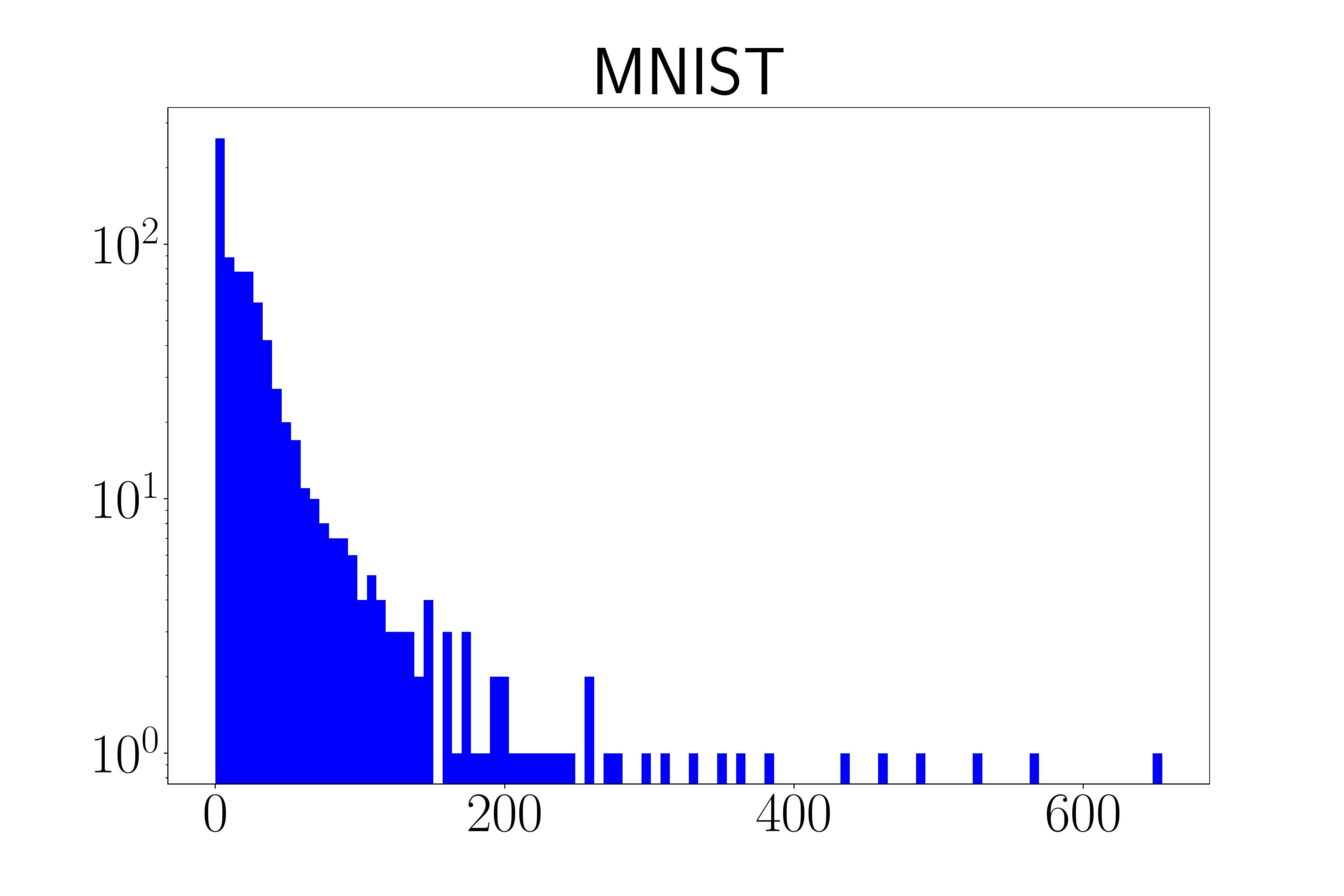}
    \hfill
    \includegraphics[width=.48\textwidth]{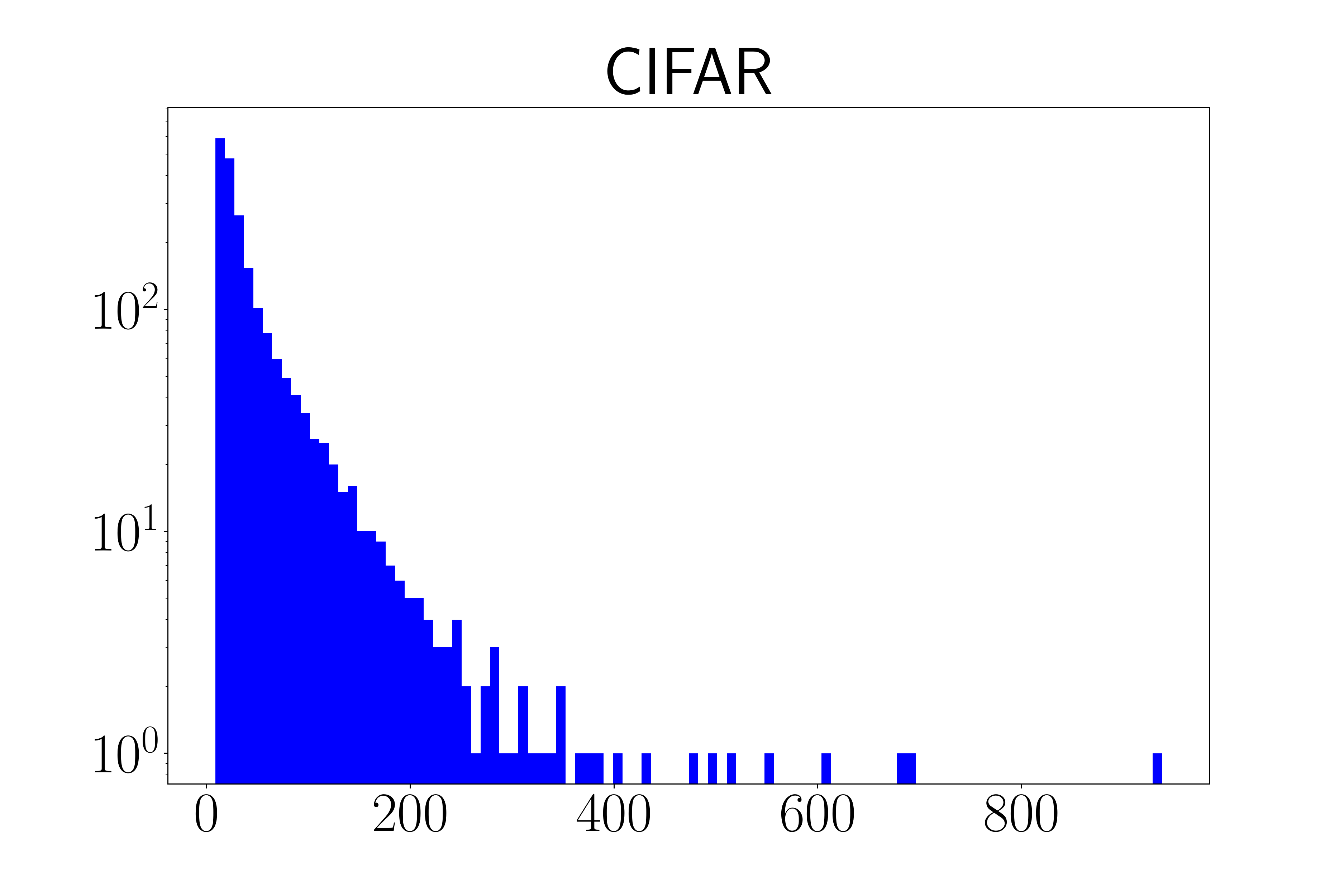}
    \caption{Empirical spectrum for the covariance
matrix of the features. On the left we considered the MNIST raw features and on the right we considered the features of an inception net applied on the CIFAR10 dataset.
}
\label{fig: real data}
\end{figure*}

\begin{figure}[t] 
    \centering
    \includegraphics[width=0.5\textwidth]{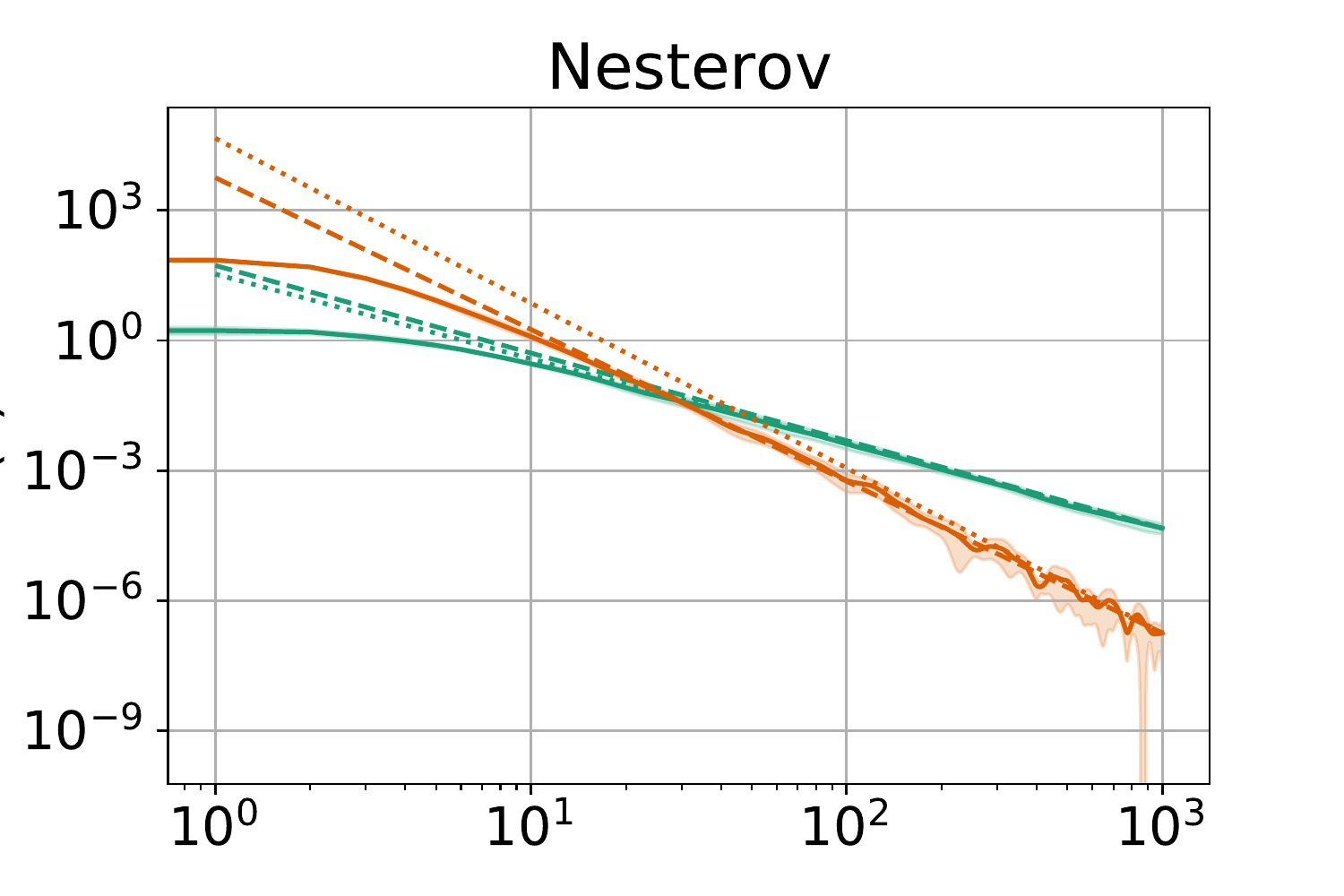}
    \includegraphics[width= 0.5\textwidth]{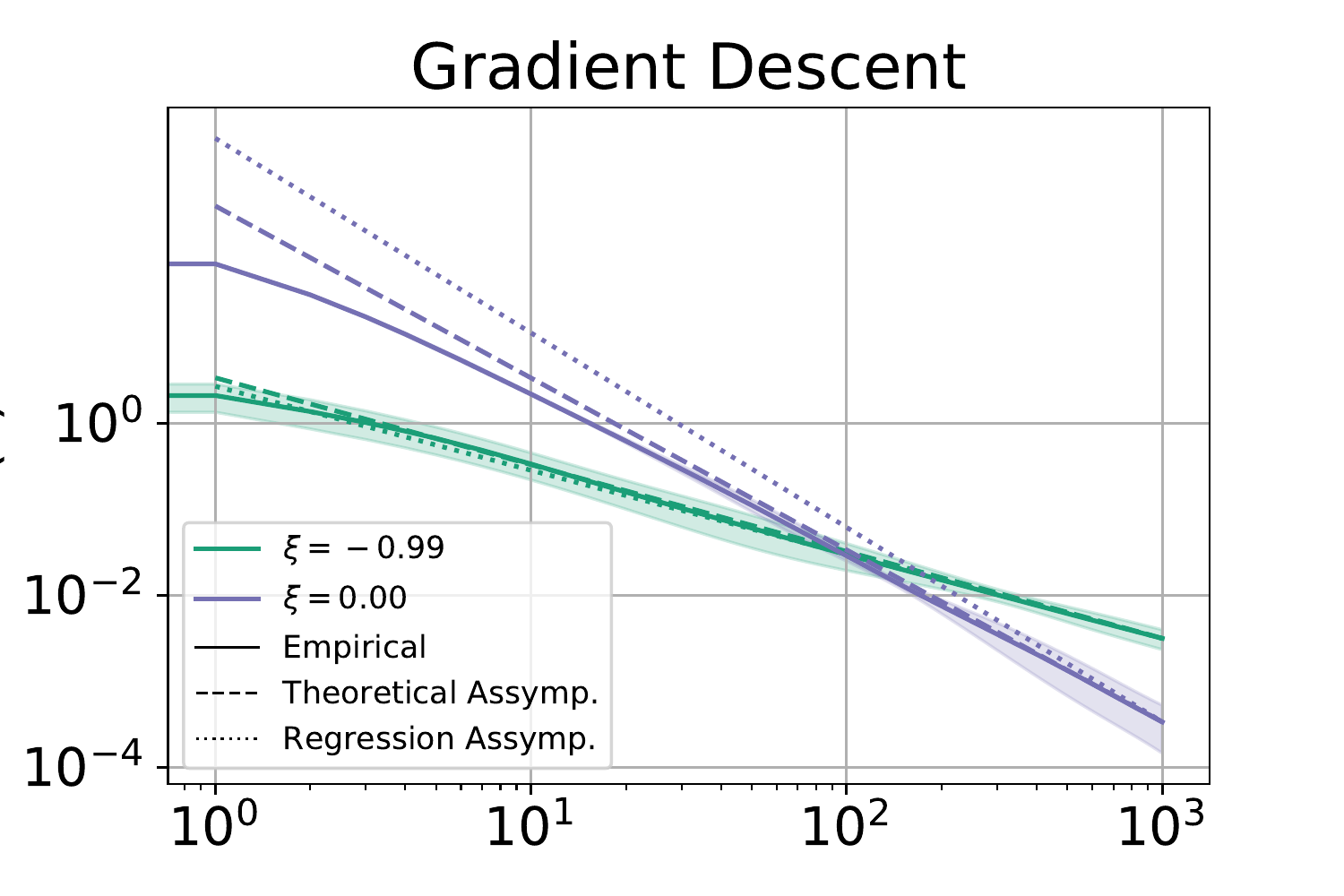}\\
        \includegraphics[width= 0.5\textwidth]{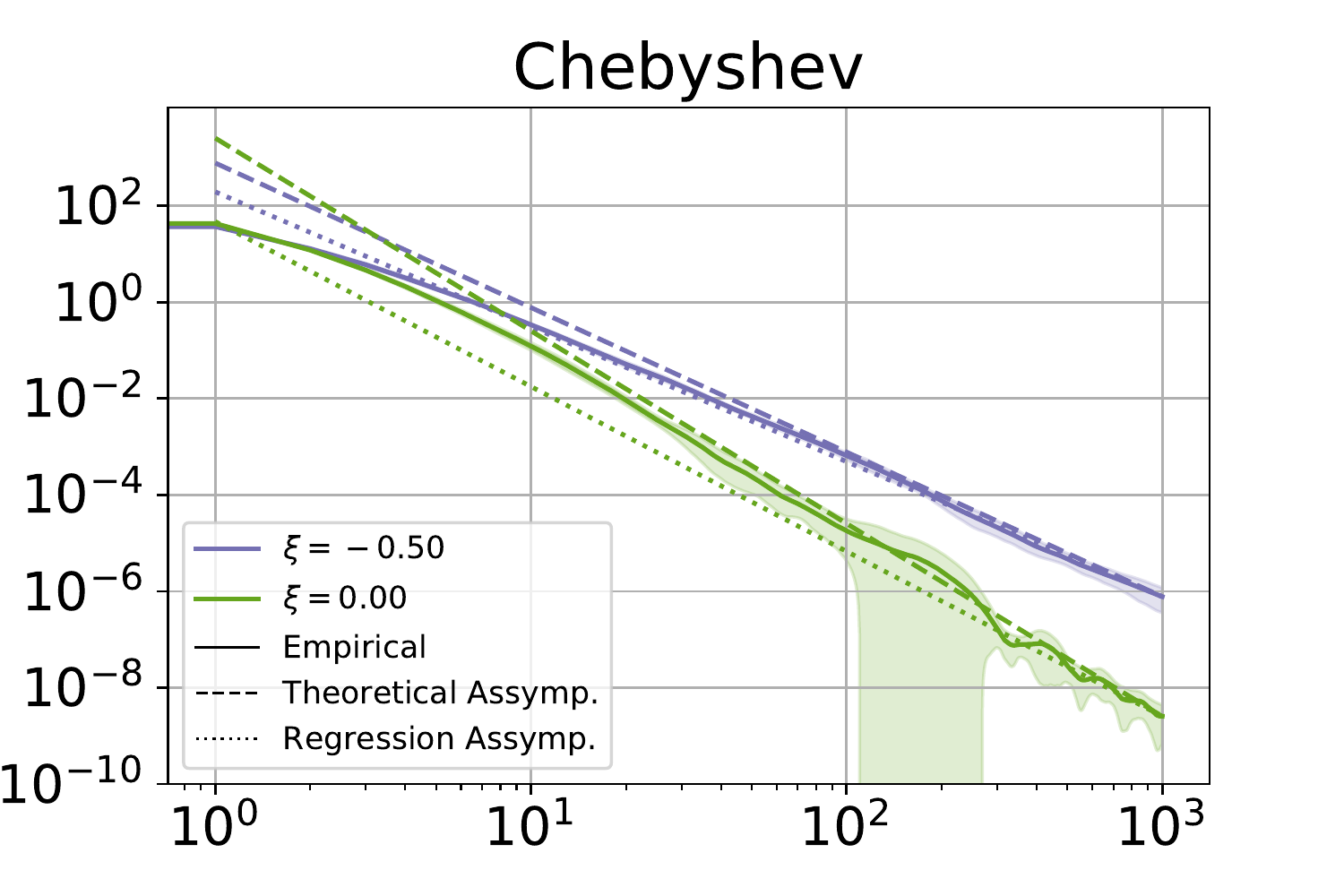}
    \caption{ 
    Comparison between experiments run on synthetic Beta distribution and theoretical asymptotic. Shades are standard deviation for 8 different samplings of the distribution and of the random initialization. Regression Assymptotic is the linear regression of the last 700 values, in logarithmic scale . Y-axis is the function value.}
    \label{fig: last figure}
\end{figure}

We generate random quadratic problems in two different ways. First, by sampling $\HH=\XX\XX^\top$ where the entries of $\XX$ are i.i.d. Gaussian. Since the ESD of $\HH$ converges to the Marchenko-Pastur distribution as the number of samples and features, go to infinity. This corresponds to a Beta distribution with parameters $(\tau,\xi) = (1/2,-1/2)$.

The other process we use to generate random quadratic problems is by sampling $\Lambda \in \RR^d$ from the corresponding Beta distribution and taking $\HH=\UU\diag(\Lambda)\UU^\top$, where $\UU$ is an independently sampled orthonormal matrix. 

We let $\xx^\star=\boldsymbol{0}$ and sample $\xx_0$ from a standard Gaussian distribution. In all experiments, we use the problem's instance largest eigenvalue to calibrate each method (e.g., gradient descent's step size is $1/L$).

Our theoretical rates in Theorem~\ref{the: gd rates} and Theorem~\ref{the: neste rates} respectively for the Nesterov method and gradient descent are precise under the approximate range $-1<\xi<0$  as we show in Figure~\ref{fig: last figure}. Distributions with higher $\xi$ need many samples. Otherwise, they behave as strongly convex functions.

The Generalized Chebyshev Method shows large variance for the different objectives  $r_t$ in some settings. Consider $\beta^\star$ the optimal $\beta$ value for a given $\xi$. If $\beta<\beta^\star$ or $\xi$ is low the function value $f(\xx_t)$, when seen as random variable w.r.t. the random initialization is slow to converge with increasing dimension. This is shown in appendix \ref{appendix: experiments}. 

The GCM with  $\beta>\beta^\star$ performs corresponding to the theory, and it is non-asymptotically very close to the performance of $\beta^\star$. High values of $\beta$ also perform very well on non-synthetic data, suggesting we should use these values in practice.

\section{Conclusion and further work}
In this paper, we've established that the asymptotic convergence of first order methods on quadratic problems in the convex regime depend on the concentration of the Hessian's eigenvalue near the edges of the spectrum's support. We further contributed to the theoretic understanding of the Nesterov's method performance and established the contrast between the worst-case and average-case in the main regimes considered in Optimization.

\subsection*{Acknowledgements}
Gauthier Gidel is supported by an IVADO grant and a Canada CIFAR AI Chair. 

\clearpage

\bibliographystyle{icml2022}
\bibliography{main.bib}

\clearpage

\appendix
\onecolumn
\section{Proofs of Section \ref{section: robust average}}

\metrics*
\begin{proof}
\newcommand\xinit{\xx_0-\xx^\star}
We remark that by the definition of the expected spectral distribution $\mu$ of $\HH$, we have for a polynomial $g$
\begin{equation}
    \EE_H[\tr(g(\HH))]=\int g(\lambda)\dif \mu(\lambda). 
\end{equation}
We know that $\xx_t-\xx^\star=P_t(\HH)(\xinit)$. We can write $\|\xx_t-\xx^\star\|^2$ in terms of a trace and use the independence of $\HH$ and $\xinit$ to connect it to the ESD:
\begin{align}
    \mathbb{E}\|\xx_t-\xx^\star\|^2&=\EE[\tr((\xinit)^\top P_t(\HH)^2(\xinit))], \\
    &=\EE_{\HH,\xinit} [\tr(P_t(\HH)^2(\xinit)(\xinit)^\top]\\
    &=\EE_\HH\left[\tr(P_t(\HH)^2\EE_{\xinit}[(\xinit)(\xinit)^\top])\right)],  \\
    &=R^2\EE_\HH[P_t(\tr(\HH))^2]=R^2 \int P_t(\lambda)^2\dif\mu(\lambda)\,.
\end{align}
For the gradient and function value the reasoning is the same by noticing that
\begin{align}
    \EE[f(\xx_t)-f(\xx^\star)]&=\frac{R^2}{2}\EE[ \tr((\xinit)^\top P_t(\HH)\HH P_t(\HH)(\xinit))]\\
    &=\EE_\HH[\tr((\lambda P_t)(\HH)^2)]=\frac{R^2}{2} \int \lambda P_t(\lambda)^2\dif\mu(\lambda),
\end{align}
where $\lambda P_t$ is also a  polynomial. As $\nabla f(\xx_t)=\HH(\xx_t-\xx^\star)$.
\begin{align}
    \EE\|\nabla f(\xx_t))\|^2&=\EE[ \tr((\xinit)^\top P_t(\HH)\HH^2 P_t(\HH)(\xinit))]\\
    &=\EE_\HH[\tr((\lambda^2 P_t)(\HH)^2)]=R^2 \int \lambda^2P_t(\lambda)^2\dif\mu(\lambda)\,.
\end{align}

\end{proof}

\optimality*
\begin{proof}
This is a direct application of the following lemma w.r.t. to the distribution $\nu$ defined as $d\nu(\lambda)=\lambda^l \dif\mu(\lambda)$.
\begin{lemma}\citep{fischer1996polynomial}
 The residual polynomial of degree $t$ that minimizes $\int P_t^2(\lambda)d\nu(\lambda)$ is given by the degree $t$ residual
orthogonal polynomial with respect to the weight function $\lambda d\nu(\lambda)$.
\end{lemma}

Let $P_t(\lambda;(a_k))$ be the polynomial of degree $k$ defined by the coefficients $a_k$ and $F((a_k))=\int \lambda^lP_t^2(\lambda;(a_k))\dif\mu(\lambda)$. The restriction $P_t(0)=1$ is expressed as $g(a_k)=a_0-1=0$. We then consider the minimization problem:
\begin{align*}
    &\min_{a_k} F(a_k) \\
    \text{s.t. } \quad &g(a_k)=0
\end{align*}
We observe that the problem is convex in the variables $a_k$ and consider the lagrangian $\mathcal{L}(a_k;\alpha)=F(a_k)-\alpha g(a_k)$.
By the first-order optimality conditions the derivatives w.r.t. the coefficients  $a_k, k>0$ of the polynomials should be $0$:
\begin{align*}
    \frac{d}{da_k}\left(\int \lambda^lP_t^2(\lambda)\dif\mu(\lambda)\right)&=\int\lambda^l\frac{\dif}{\dif a_k}\left(\sum_{k=0}^\top a_k\lambda^kP_t(\lambda) \right) \dif \mu(\lambda),
    \\&=2\cdot\left(\int \lambda^{l+k}P_t(\lambda)\dif\mu(\lambda)\right)=0.
\end{align*}
This means that $P_t(\lambda)$ is orthogonal to any polynomial of degree $t-1$  w.r.t to the intern product $\langle.,.\rangle_{\lambda^{l+1}\dif\mu}$

\end{proof}

\section{GCM and Laguerre method derivation}
We will first state two lemmas that allow us to construct the optimal polynomials.
%With them in hand the procedure is trivial.

\begin{lemma} \label{lemma: to residual}
Let $(\tp_t)$ be a family polynomials with recurrence  
\begin{equation*}
    \tp_t(\lambda)=(\alpha_t+\beta_t\lambda)\tp_{t-1}(\lambda)+\gamma_t\tp_{t-2}(\lambda),
\end{equation*}
with $\tp_0$ a constant polynomial and $\tp_t(0)\neq0$ for all $t$. Then,
\begin{equation}
    P_t(\lambda)=(a_t+b_t\lambda)P_{t-1}(\lambda)+(1-a_t)P_{t-2}(\lambda)\,,
\end{equation} is the recurrence for $P_t(\lambda)=\tp_t(\lambda)/\tp_t(0)$. With:
\begin{align}
    a_t&=\delta_t\alpha_t, \\
    b_t&=\delta_t\beta_t, \\
    \delta_t&=(\alpha_t+\gamma_t\delta_{t-1}) \hspace{0.5 cm } (\delta_0=0).
\end{align}
\end{lemma}
The proof of this is presented in \cite{pedregosa2020acceleration}. Further, we know how to compute the recurrence for the polynomials of a shifted distribution:
\begin{lemma}

Let $(\tp_t)$ be a family polynomials   following  
\begin{equation}    \label{rec}
    \tp_t(\lambda)=(\alpha_t+\beta_t\lambda)\tp_{t-1}(\lambda)+\gamma_t\tp_{t-2}(\lambda) ,
\end{equation}
and define polynomials $P_t$ such that :
$$
P_t(m(\lambda))=\tp_t(\lambda),
$$
with $m(\lambda)=a\lambda+b$ a non-singular affine transform. Then $P_t$ follows a recurrence like in eq. \eqref{rec}, with:
\begin{align}
    \alpha_t'&=\alpha_t+b\beta_t, \\
    \beta'_t&=a\beta_t\,,\\
    \gamma'_t&=\gamma_t\,.
\end{align}
\end{lemma}
\begin{proof}
The result follows from considering Eq. \eqref{rec} with argument $m^{-1}(\lambda)$.
\label{jacobi recurrence}
\end{proof}

These results are sufficient to obtain the recurrence relation for the residual polynomial w.r.t $x^\beta(L-x)^\alpha$. We begin by the standard Jacobi polynomials, which are orthogonal w.r.t $(1-x)^\alpha(1+x)^\beta$ and follow a recurrence according to $\alpha_t,\beta_t,\gamma_t$ below \cite{szego1975orthogonal}:
\begin{align}
    \alpha_t&=\frac{(2n+\alpha+\beta)(2n+\alpha+\beta-1)}{2n(n+\alpha+\beta)},\\
    \beta_t&=\frac{(\alpha^2-\beta^2)(2n+\alpha+\beta-1)}{2n(n+\alpha+\beta)(2n+\alpha+\beta-2)},\\
    \gamma_t&=\frac{-2(n+\alpha-1)(n+\beta-1)(2n+\alpha+\beta)}{2n(n+\alpha+\beta)(2n+\alpha+\beta-2)}\,.
\end{align}
We then shift the distribution according to $\eta(x)$, and then transform to the residual ones. We slightly simplify these computations and use remark \ref{rmk: momentum based} to get Algorithm \ref{algo: chebyshev}.

We know \citep{szego1975orthogonal} that the Laguerre polynomials $L_t^\alpha$, with usual normalization, follow the recurrence

\begin{equation}
    L_t^\alpha(\lambda)=\left(\frac{2t+\alpha-1}{t}-\frac{1}{t}\lambda\right)L_{t-1}^\alpha(\lambda)+\frac{t+\alpha-1}{t}L_{t-2}^\alpha(\lambda)\,.
\end{equation}
As we don't have to shift the domains, we have only to apply lemma \ref{lemma: to residual} to get the Laguerre method. Further, we can get a explicit expression for $\delta_t=\frac{t}{t+\alpha}$ by simplifying the expression.

\begin{algorithm}
\caption{Laguerre($\alpha$)}
\begin{algorithmic}
\STATE \textbf{Inputs}: Initial vector $\xx_0$, function $f$, distributional parameter $\alpha$ \\
$\xx_{-1}\gets \textbf{0}$\\
\FOR{$t=1,\ldots,T$}
\STATE $\xx_t\gets\xx_{t-1}+\frac{t-1}{t+\alpha}(\xx_{t-1}-\xx_{t-2})-\frac{1}{t+\alpha}\nabla f(\xx_{t-1})$
\ENDFOR
\end{algorithmic}
\end{algorithm}

\section{Proofs of section \ref{section: robust average}}
In the following we will consider shifted versions of the spectral distributions. This shift is the affine function $m(\lambda):[0,L]\rightarrow [-1,1]$ because most results in the theory of orthogonal polynomials are stated in terms of distributions supported in $[-1,1]$.

This can be seen as an additional layer of abstraction because the quantities evaluated with the shifted distributions and polynomials are proportional, i.e. if $P_t(x)=\tp_t(m(x))$ and $\mu'(x)=\tmu'(m(x))$, then,
\begin{equation}
    \int P_t^2(x)\mu'(x)\dif x\propto \int \tp_t^2(x)\tmu'(x)\dif x,
\end{equation}
so all the asymptotics are the same. The Jacobi polynomials $\JP_t$ are those orthogonal w.r.t $\dif\mu(x)=(1-x)^\alpha(1+x)^\beta \dif x$. Most works use the  normalization $\tp_t^{\alpha,\beta}(-1)=(-1)^t\binom{t+\beta}{t}$. We will write $\tp^{\alpha,\beta}_t$ for this normalization and $\JP_t$ for the residual polynomials
\robustjacobi*
\begin{proof}
We will prove that for any $\alpha$ and $\beta$, $\xi,\tau>-1$, $l>0$ and distributions $\nu_{\tau,\xi-l}$, we have
\begin{equation*}
    \int P_t^{\alpha,\beta}(x)^2x^l d\nu_{\tau,\xi-l}(x) \sim  L^lC^{\alpha,\beta}_\nu\left\{
	\begin{array}{ll}
		  t^{-1-2\beta}& \mbox{if } 
		  \alpha<\tau+1/2 \text{ and } \beta <\xi+1/2,\\
		  t^{-2(\xi+1)}\log t& \mbox{if } 
		  \alpha=\tau+1/2 \text{ and } \beta =\xi+1/2,\\
		  t^{2(\max\{\alpha-\beta-\tau,-\xi\}-1)}& \mbox{if } 
		  \alpha>\tau+1/2 \text{ or } \beta >\xi+1/2.
	\end{array}
\right.
\end{equation*}
We will first show this result for the beta weights, then show that distributions with the same concentration behave similarly. \\
The normalization of $\tp^{\alpha,\beta}_t$ is s.t. \cite{szego1975orthogonal} (4.3.3):
\begin{equation}
    \int_{-1}^1\tp_t^{\alpha,\beta}(x)(1-x)^\alpha(1+x)^\beta \dif x=\frac{2^{\alpha+\beta+1}}{2t+\alpha+\beta+1}\frac{\Gamma(t+\alpha+1)\Gamma(t+\beta+1)}{\Gamma(t+1)\Gamma(t+\alpha+\beta+1)}=\Theta(t^{-1})\,.
\end{equation}

Further, the residual polynomials  are s.t. $|P^{\alpha,\beta}_t|=\Theta(t^{-\beta})|\tp^{\alpha,\beta}_t|$, from the definition of the classical normalization.\\
We state the result (Exercise 91, Generalisation of 7.34.1) from \cite{szego1975orthogonal}:
\begin{lemma}[\citet{szego1975orthogonal}] \label{assymp. lemma}
We have,
\begin{align}
    &\int_0^1(1-x)^\tau P_t^{\alpha,\beta}(x)^2dx \sim\Theta( h_{\tau}^\alpha),\\
    &h_{\tau}^\alpha\defas
\left\{
	\begin{array}{ll}
		t^{2(\alpha-\tau-1)}  & \mbox{if } \alpha>\tau+1/2, \\
		t^{-1}\log t   & \mbox{if } \alpha=\tau+1/2, \\
		t^{-1}   & \mbox{if } \alpha<\tau+1/2 \, .
	\end{array}
\right.
\end{align}
 \label{jacobi lemma}
\end{lemma}
Noting that $\tp^{\alpha,\beta}_t(x)=(-1)^\top\tp_t^{\beta,\alpha}(-x)$, we can write:
\begin{equation}
    \int_{-1}^1\tp_t(x)^2(1-x)^\tau(1+x)^\xi dx = \Theta\left(\int_0^1(1-x)^\tau|\tp_t^{\alpha,\beta}(x)|^2dx\right) +\Theta\left(\int_0^1(1-x)^\xi|\tp_t^{\beta,\alpha}(x)|^2dx\right). \label{int decomp}
\end{equation}
We can then show our result for $\dif\nu_{\tau,\xi-l}(x)=x^{\xi-l}(L-x)^\alpha \dif x$ by carefully considering each of the cases on Lemma~\ref{jacobi lemma} and the maximum of each term in eq. \ref{int decomp}, and an added $t^{-2\beta}$ from the different normalization. With this, we have the wanted result for the Beta weights \\
It remains to show:
\begin{equation}
    \int_0^1 \tp_t^{\alpha,\beta}(x)^2\dif\nu_{\tau,\xi}(x)= \Theta\left(\int_0^1(1-x)^\tau\tp_t^{\alpha,\beta}(x)^2dx \right).
\end{equation}
And the rest follows from the same arguments. We do this with the help of this lemma shown in \cite{van1995weak} relating to the weak convergence of the orthogonal polynomials:

\begin{lemma}[\citet{van1995weak}]
 Let $\mu$ be a measure and $(p_t)$ a family of associated orthonormal polynomials such that $p_t$ follow the recurrence:
 \begin{equation*}
     xp_t(x)=a_tp_{t+1}(x)+b_tp_t(x)+a_{t-1}p_{t-1}(x),
 \end{equation*}
 and $a_t,b_t$ converge respectively to $a,b$. Then for any $f$  continuous and bounded we have:
\begin{equation}
    \int f(x)p_t^2(x)\dif\mu(x) \rightarrow \frac{1}{\pi}\int_{-1}^1 \frac{f(x)}{\sqrt{1-x^2}}dx \,.
\end{equation}
\label{wk}
\end{lemma} 
 Let $\epsilon$ such that
\begin{equation}
    x\geq 1-\epsilon \Rightarrow |\dif\nu_{\tau,\xi}-A(1-x)^\tau|\leq  B(1-x)^\tau \label{eq: epsilon}.
\end{equation}
We observe that for $0<x<1-\epsilon$, $f(x)=\frac{d\nu_{\tau,\xi}}{(1-x)^\alpha(1+x)^\beta}$ is bounded. \\
We get from an application of \ref{wk}, and the observation that $\tp_t^{\alpha,\beta}=\mathcal{N}_tp_t^{\alpha,\beta}$, with constants $\mathcal{N}_t=\Theta(t^{-1/2})$:
\begin{align}
    \underbrace{\int_0^1(1-x)^\tau\tp_t^{\alpha,\beta}(x)^2\dif x}_{\Theta( h_\tau^\alpha)}&=\underbrace{\int_0^{1-\epsilon}(1-x)^\tau\tp_t^{\alpha,\beta}(x)^2\dif x}_{\Theta( t^{-1})
    } +\int_{1-\epsilon}^1(1-x)^\tau\tp_t^{\alpha,\beta}(x)^2\dif x \Rightarrow\\
    &\int_{1-\epsilon}^1(1-x)^\tau\tp_t^{\alpha,\beta}(x)^2\dif x  =\Theta( h_\tau^\alpha),
    \end{align}
\begin{equation}
    \begin{split}
    \int_0^1 \tp_t^{\alpha,\beta}(x)^2\dif\nu_{\tau,\xi}(x)&=
    \underbrace{\int_0^{1-\epsilon} \tp_t^{\alpha,\beta}(x)^2 f(x) (1-x)^\alpha(1+x)^\beta\dif x)}_{\Theta( t^{-1})
    }\\
    &+\Theta\left(
    \underbrace{\int_{1-\epsilon}^1(1-x)^\tau\tp_t^{\alpha,\beta}(x)^2\dif x}_{\Theta( h_\tau^\alpha)}\right).
    \end{split}
\end{equation}

\end{proof}

\jacoptimal*
\begin{proof}
We will prove that for $\tau,\xi>-1$ If $\alpha = \tau$ and $\beta = \xi+l+1$ (i.e., are optimal), the rate of convergence reads
\begin{equation}
    \min_{P_t(0)=1}\int P_t^2(\lambda)\lambda^ld\nu(\lambda)=\Theta\left( \int_{0}^l  \tp_t^{\alpha,\beta}(\lambda)^2(L-\lambda)^\tau\lambda^{\xi+l}\dif \lambda\right) =\Theta( t^{-2(\xi+l+1)}).
\end{equation}
Showing the second equality is easy by considering theorem \ref{thm: jacobirates}, and that is further the minimum asymptotic rate for the Beta distribution $\mu_{\tau,\xi}$. \\
%As $\tp^{\alpha,\beta}_t$ has the same rate on $\nu$ as  on the Beta distribution , the minimum rate for $\nu$ is lower bounded by the r.h.s. \\
By setting $p_t^\nu$ and $P^\nu_t=\frac{p_t^\nu}{p^\nu_t(0)}$ the optimal orthonormal and residual and  polynomials w.r.t. $\nu$ we show that $P^\nu_t$ must have the same rate on $\mu_{\tau,\xi}$ as it does on $\nu$, thus the optimal rate of $\nu$ cannot be lower than the optimal rate of $\mu_{\tau,\xi}$. Indeed, setting $\epsilon_1,\epsilon_2$ as in eq. \ref{eq: epsilon}:
\begin{align}
    \int_{1-\epsilon_2}^{1}P_t^\nu(x)^2d\nu(x)&=\Theta\left(\int_{1-\epsilon_2}^{1}P_t^\nu(x)^2\dif\mu_{\tau,\xi}(x)\right),\\
    \int_{-1}^{-1+\epsilon_1}P_t^\nu(x)^2d\nu(x)&=\Theta\left(\int_{-1}^{-1+\epsilon_1}P_t^\nu(x)^2\dif\mu_{\tau,\xi}(x)\right),\\
    \int_{-1+\epsilon_1}^{1-\epsilon_2}P_t^\nu(x)^2d\nu(x)&=\Theta\left(\int_{-1+\epsilon_1}^{1-\epsilon_2}P_t^\nu(x)^2\dif\mu_{\tau,\xi}(x)\right)=\Theta\left(\frac{1}{p_t^\nu(-1)^2}\right),
\end{align}
where the first two equations come from the fact that $\nu=\Theta(\mu_{\tau,\xi})$ near $-1$ and $1$ and the third from lemma \ref{wk}.\\
This effectively lower bounds the rates on $\nu$ because the rates of $P_t^\nu$ on $\mu_{\tau,\xi}$ can't be lower than $-2(\xi+l+1)$.
\end{proof}

\worstcase *
\begin{proof}
We will prove that:  $\sup_{x\in[0,L]}x^lP_t^{\alpha,\beta}(x)^2=O(L^lt^{v(\alpha,\beta,l)})$, where:
\begin{equation}
    v(\alpha,\beta,l)=\left\{
    \begin{array}{cc}
           2(\alpha-\beta) &\text{if} \hspace{0.5 cm} \alpha>\beta-l \\
         -1-2\beta \hspace{1 cm} &\text{if} \hspace{0.5 cm} \alpha\leq \beta-l\hspace{0.5 cm} \beta\leq l-\frac{1}{2},\\
         -2l, \hspace{1 cm} &\text{if} \hspace{0.5 cm} \alpha\leq\beta-l\hspace{0.5 cm} \beta\geq l-\frac{1}{2} .
    \end{array}
    \right . 
\end{equation}
From \cite{szego1975orthogonal}, Theorem 7.32.2, if $\theta<\frac{\pi}{2}$:
\begin{equation}
    \tp_t^{\alpha,\beta}(\cos \theta)=\left\{ 
    % \begin{array}{cc}
    %      O(t^{-1/2})  \hspace{1 cm} &\text{if} \hspace{0.5 cm} \alpha < -\frac{1}{2}, \\
    %      O(t^{\alpha})  \hspace{1 cm} &\text{if} \hspace{0.5 cm} \alpha \geq -\frac{1}{2} , 0\leq\theta\leq ct^{-1}, \\
    %      \theta^{-
    %      \alpha-1/2}O(t^{-1/2})  \hspace{1 cm} &\text{if} \hspace{0.5 cm} \alpha \geq -\frac{1}{2} , \theta> ct^{-1}.
    % \end{array}
    \begin{array}{cc}
         O(t^{-1/2})  \quad &\text{if} \hspace{0.5 cm} \alpha < -\frac{1}{2}, \\
         O(t^{\alpha})  \quad &\text{if} \hspace{0.5 cm} \alpha \geq -\frac{1}{2} \land 0\leq\theta\leq ct^{-1}, \\
         \theta^{-
         \alpha-1/2}O(t^{-1/2})  \quad&\text{if} \hspace{0.5 cm} \alpha \geq -\frac{1}{2} \land \theta> ct^{-1}.
    \end{array}
    \right .
\label{lemma: worst case}
\end{equation}
We observe that, from the symmetry of the Jacobi polynomials:
\begin{equation}
    \sup_{x\in[0,L]}x^lP_t^{\alpha,\beta}(x)^2 =\Theta\left( \max\left\{\sup_{x\in[0,1]}x \tp_t^{\alpha,\beta}(x)^2,\sup_{x\in[0,1]}(1-x)^l\tp_t^{\beta,\alpha}(x)^2\right\}\right).
\end{equation}
The $(1-x)^l$ term,  corresponds to $(2\sin(\frac{\theta}{2}))^{2l}$ in the variable $\theta$, which is $O(\theta^{2l})$. The rest follows from carefully considering the expressions given by eq. \ref{lemma: worst case}.
\end{proof}

\nesterovrates *

\begin{proof}
We will prove:
\begin{equation}
    \int_0^1P_t^{\text{Nes}}(\lambda)^2\lambda^l\dif\nu_{\tau,\xi-l}\sim C'_\nu
    \Big\{\begin{array}{ll}
          t^{-2(\xi+1)}& \mbox{if } 
		  0<\xi<1/2,  \\
		  t^{-3}\log t& \mbox{if } 
		  \xi=1/2,\\
		  t^{-(\xi+5/2)}& \mbox{if } 
		  \xi>1/2.
	\end{array}
\end{equation}
\cite{paquette2020halting} has shown that the nesterov polynomials $P_t$ are asymptotically, in $t$:
\begin{equation}
    P_t(\lambda)\sim\frac{2J_1(t\sqrt{\alpha\lambda})}{t\sqrt{\alpha\lambda}}e^{-\alpha\lambda t/2},
\end{equation}
in the sense that:
\begin{equation}
    \int_0^1u^{l}\left[\Tilde{P_t^2(u)}-\frac{4J_1^2(t\sqrt{u})}{t^2u}e^{-u t}\right]\dif\mu_{MP}(u)=O(t^{-(l+25/12))}).
\end{equation}
The arguments can be easily used to show that such an integral is $O(t^{ -(\alpha+l+31/12)})$ when evaluated w.r.t. a general $\dif\mu$ s.t $\mu'=\Theta(\lambda^\alpha)$ near $0$. \\
We can thus consider our integral  of interest substituting $P_t^\text{Nes}$ by it's Bessel asymptotic and dividing it into three regions, i.e. $[0,1]=[0,\frac{\epsilon}{t}]\cup[\frac{\epsilon}{t},\frac{\epsilon}{\sqrt{t}}]\cup[\frac{\epsilon}{\sqrt{t}},1]$ corresponding to two different regimes for the Bessel function. The first region will give us the asymptotic and  we will bound the others.\\
We consider first, the "middle region",i .e. for some $\epsilon>0$:
\begin{equation}
    \int_{\frac{\epsilon}{t}}^{\frac{\epsilon}{\sqrt{t}}} u^{\xi}\frac{4J_1^2(t\sqrt{u})}{t^2u}e^{-u t}\dif u.
\end{equation}
We note the asymptotic for $J_1^2$:
\begin{equation}
    J_1^2(\sqrt{tv}) \sim \frac{1}{\pi\sqrt{tv}}(1+\cos(2\sqrt{tv}+2\gamma)).
\end{equation}
Doing the change of variable $v=tu$, and identifying the upper limit of the interval, which is $\epsilon t^{1/2}$, to $\infty$:

\begin{align}
    \int_{\frac{\epsilon}{t}}^{\frac{\epsilon}{\sqrt{t}}} u^{\xi}\frac{4J_1^2(t\sqrt{u})}{t^2u}e^{-u t}\dif u &=\Theta\left(
    t^{-2-\xi}\int_\epsilon^\infty v^{\xi-1}J_1^2(\sqrt{tv})e^{-v}\dif v\right),\\
    &=\Theta\left( t^{-2-\xi}\int_\epsilon^\infty v^{\xi-1}\frac{1}{\pi\sqrt{tv}}e^{-v}\dif v \right),\\
    &=\Theta\left(t^{-\frac{5}{2}-\xi}\underbrace{\int_\epsilon^\infty v^{\xi-\frac{3}{2}}\frac{1}{\pi\sqrt{tv}}e^{-v}\dif v}_{\Gamma(\xi-\frac{1}{2},\epsilon) }\right).
\end{align}
Where, from the Riemann-Lebesgue, the cosine term goes to $0$  lemma and $\Gamma$ is the incomplete Gamma function.\\
The mass from the rightmost region ,i.e, the term corresponding to the interval $[\epsilon t^{-1/2},1]$ is exponentially small. Indeed, because of the exponential $e^{-ut}$ it is  $O(e^{-\epsilon\sqrt{t}})$. \
Lastly, we have for the $[0,\frac{\epsilon}{t}]$ region, doing the change of variables $v=t^2u$:
\begin{equation}
    \int_0^{\frac{\epsilon}{t}} u^{\xi}\frac{4J_1^2(t\sqrt{u})}{t^2u}e^{-u t}\dif u =\Theta\left(
    t^{-2(\xi+1)}\int_0^{t\epsilon} v^{\xi}\frac{J_1^2(\sqrt{v})}{v}e^{-\frac{v}{t}}\dif v\right).
\end{equation}
And the $e^{\frac{-v}{t}}$ term  is $\Theta(1)$. We have the following Bessel asymptotics:
\begin{align}
    \frac{J_1^2(\sqrt{v})}{v}&\sim \frac{1}{4}, \hspace{2 cm} v\rightarrow 0, \\
    \frac{J_1^2(\sqrt{v})}{v}&= O(v^{-3/2}), \hspace{1.0 cm} v\rightarrow \infty,
\end{align}
so we divide this integral aswell:

\begin{align}
    t^{-2(\xi+1)}\int_1^{t\epsilon} v^{\xi}\frac{J_1^2(\sqrt{v})}{v}e^{-\frac{v}{t}}\dif v,
    &=\Theta\left(t^{-2(\xi+1)}\int_\epsilon^{t\epsilon} v^{\xi-\frac{3}{2}}\dif v\right) =\Theta\left( I_\xi(t)t^{-\xi-\frac{5}{2}}\right)\\
    t^{-2(\xi+1)}\int_0^{1} v^{\xi}\frac{J_1^2(\sqrt{v})}{v}e^{-\frac{v}{t}}\dif v
    &=\Theta\left(t^{-2(\xi+1)}\int_0\epsilon^{1} v^{\xi} \dif v\right) =\Theta\left( t^{-2(\xi+1)}\right),
\end{align}
where $I_\xi(t)=\log t$ if $\xi=\frac{1}{2}$ and $1$ otherwise. \\
The Nesterov rate is then $I_\xi(t)t^{-\xi-\frac{5}{2}}$ if $\xi\geq\frac{1}{2}$ and $t^{-2(\xi+1)}$ if $0<\xi<\frac{1}{2}$
\end{proof}
\gdrates*
\begin{proof}
Considering that $P_t^\text{GD}(\lambda)=(1-\frac{\lambda}{L})^\top$ we will prove :
\begin{equation}
    \int_0^1(1-\lambda)^{2t}\lambda^l\dif\nu_{\tau,\xi-l}=\Theta(t^{-(\xi+l+1)}).
\end{equation}
We know, for the Beta weights, that:
\begin{equation}
    \int_0^1(1-\lambda)^{2t+\tau}\lambda^{\xi+l}\dif\lambda=\frac{\Gamma(l+\xi+1)\Gamma(2t+\tau+1)}{\Gamma(2t+l+\xi+\tau+2)}=\Theta(t^{-(\xi+l+1)}). \label{eq: closed form}
\end{equation}
We can identify this asymptotic to the interval $\int_0^\epsilon$ for any $\epsilon$ because:
\begin{equation}
    \int_\epsilon^1(1-\lambda)^{2t+\tau}\lambda^{\xi+l}\dif\lambda=\mathcal{O}((1-\epsilon)^{2t}),
\end{equation}
then:
\begin{align}
    \int_\epsilon^1(1-\lambda)^{2t}\lambda^l\dif\nu_{\tau,\xi-l}&=\mathcal{O}((1-\epsilon)^{2t}), \\
    \int_0^\epsilon(1-\lambda)^{2t}\lambda^l\dif\nu_{\tau,\xi-l}&=\Theta\left(\int_0^\epsilon(1-\lambda)^{2t+\tau}\lambda^{\xi+l}\dif\lambda\right)=\Theta(t^{-(\xi+l+1)}).
\end{align}

\end{proof}
\laguerrerates*
\begin{proof}
Let $L_t^\alpha$ be the Laguerre polynomials with the usual normalization \cite{szego1975orthogonal}:
\begin{equation}
    \int L_t^\alpha(x)^2\dif\mu_\alpha(x)=L_t^\alpha(0)=\binom{n+\alpha}{n} 
\end{equation}
Further [\cite{szego1975orthogonal} (5.1.13)]]:
\begin{equation}
    \sum_{k=0}^T L_t^\alpha(x)=L_t^{\alpha+1}(x).
\end{equation}
Thus, letting $P_t^\alpha$ be the residual Laguerre polynomial, we consider:
\begin{equation}
\begin{split}
    \EE[f(\xx_t)-f(\xx^\star)]&=\int P_t^{\alpha+2}(\lambda)^2\dif\mu_{\alpha+1}(\lambda)=\binom{t+\alpha+2}{t}^{-2}\int L_t^{\alpha+2}\dif\mu_{\alpha+1}(\lambda),\\
    &=\binom{t+\alpha+2}{t}^{-2}\sum_{k=0}^{t}\left[\int L_k^{\alpha+1}(\lambda)\dif\mu_{\alpha+1}(\lambda)\right],\\
    &=\binom{t+\alpha+2}{t}^{-2}\sum_{k=0}^\top\binom{k+\alpha+1}{k}=\binom{t+\alpha+2}{t}^{-2}\binom{t+\alpha+2}{t},\\
    &=\binom{t+\alpha+2}{t}^{-1}=\Theta( t^{-(\alpha+2)}).
\end{split}
\end{equation}
\end{proof}
\section{Additional Experiments} \label{appendix: experiments}

\begin{figure}[H]
    \centering
    
    \includegraphics[width= 0.35 \textwidth]{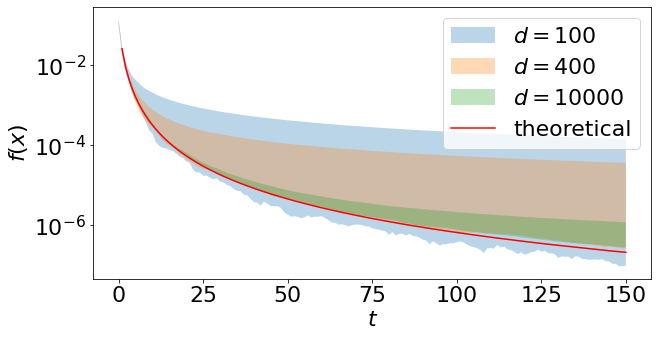}
    \caption{Function value for $\beta<\beta^\star$ for different dimensions of the problem}
    \label{fig:my_label}
\end{figure}
We observed that in the regimes $\beta<\beta^\star$ and $\xi\approx-1$ the objectives  exhibit much higher variance than outside these regimes, where they concentrate tightly around the expected value. Despite this the objectives still converge with increasing dimension.
\end{document}